\documentclass[final,12pt]{colt2023} 

\usepackage{mathtools}
\usepackage{comment}
\usepackage{xargs}

\def\msx{\mathsf{X}}

\def\cost{\mathsf{c}}
\def\rset{\mathbb{R}}
\def\psiMonge{\Psi}
\def\TMonge{\mathrm{T}}
\def\ie{\textit{i.e.}}

\usepackage{cleveref}

\crefformat{assumptionH}{{\textbf{H}}#2#1#3}

\newcommand{\bbRD}{\mathbb{R}^d}

\newcommand{\cP}{\mathcal{P}}

\newcommand{\bes}{\begin{equation}}
\newcommand{\ees}{\end{equation}}
\newcommand{\beas}{\begin{eqnarray}}
\newcommand{\eeas}{\end{eqnarray}}
\newcommand{\bea}{\begin{eqnarray}}
\newcommand{\eea}{\end{eqnarray}}
\newcommand{\be}{\begin{equation}}
\newcommand{\ee}{\end{equation}}
\newcommand{\bei}{\begin{itemize}}
\newcommand{\eei}{\end{itemize}}
\newcommand{\bec}{\begin{cases}}
\newcommand{\eec}{\end{cases}}
\newcommand{\ben}{\begin{enumerate}}
\newcommand{\een}{\end{enumerate}}

\newcommand{\bbE}{\mathbb{E}}

\newcommand{\bbl}{\begin{block}}
\newcommand{\ebl}{\end{block}}

\newcommand{\De}{\mathrm{d}}

\newcommand{\rmP}{\mathrm{P}}

\newcommand{\rmR}{\mathrm{R}}
\newcommand{\rme}{\mathrm{e}}

\newtheorem{assumption}[theorem]{Assumption}

\newcommand{\mm}{\mathbf{m}}

\newcommand{\sfd}{{\sf d}}

\def\torus{\tset}
\def\torusL{\torus_L}
\def\Haar{\mathtt{H}}
\def\HaarL{\Haar_L}

\newcommandx{\Vnorm}[2][1=V]{\| #2 \|_{#1}}
\newcommandx{\norm}[2][1=]{\ifthenelse{\equal{#1}{}}{\left\Vert #2 \right\Vert}{\left\Vert #2 \right\Vert^{#1}}}
\newcommandx{\normLigne}[2][1=]{\ifthenelse{\equal{#1}{}}{\Vert #2 \Vert}{\Vert #2\Vert^{#1}}}

\newcommand{\defEns}[1]{\left\lbrace #1 \right\rbrace }

\def\mrl{\mathrm{L}}
\def\rml{\mrl}

\def\rset{\mathbb{R}}
\def\tset{\mathbb{T}}

\def\zset{\mathbb{Z}}
\def\nset{\mathbb{N}}

\def\ie{\textit{i.e.}}

\def\eqsp{\;}

\def\bfm{\mathbf{m}}

\def\rmd{\mathrm{d}}
\def\Ent{\mathrm{Ent}}
\def\Leb{\mathrm{Leb}}
\def\varphistar{\varphi^\star}
\def\psistar{\psi^{\star}}
\def\pistar{\pi^{\star}}

\def\cconst{\mathrm{c}_{\mathrm{S}}}

\def\PP{\mathbb{P}}
\def\mcf{\mathcal{F}}
\def\blambda{\bar{\lambda}}
\def\blambdaV{\blambda_V}
\def\msm{\mathsf{M}}

\usepackage{autonum}

\title[Exponential convergence of Sinkhorn algorithm: a coupling approach]{Non-asymptotic convergence bounds for Sinkhorn iterates and their gradients: a coupling approach.}
\usepackage{times}

\coltauthor{%
 \Name{Giacomo Greco} \Email{g.greco@tue.nl}\\
 \addr Eindhoven University of Technology
 \AND
 \Name{Maxence Noble} \Email{maxence.noble-bourillot@polytechnique.edu}\\
 \Name{Giovanni Conforti} \Email{giovanni.conforti@polytechnique.edu}\\
 \Name{Alain Oliviero--Durmus} \Email{alain.durmus@polytechnique.edu}\\
 \addr \'Ecole Polytechnique
}

\begin{document}

\maketitle

\begin{abstract}%
    Computational optimal transport (OT) has recently emerged as a powerful framework with applications in various fields. In this paper we focus on a relaxation of the original OT problem, the entropic OT problem, which allows to implement efficient and practical algorithmic solutions, even in high dimensional settings. This formulation, also known as the Schr\"odinger Bridge problem, notably connects with Stochastic Optimal Control (SOC) and can be solved with the popular Sinkhorn algorithm. In the case of discrete-state spaces, this algorithm is known to have exponential convergence; however, achieving a similar rate of convergence in a more general setting is still an active area of research. In this work, we analyze the convergence of the Sinkhorn algorithm for probability measures defined on the $d$-dimensional torus $\torusL^d$, that admit densities with respect to the Haar measure of $\torusL^d$. In particular, we prove pointwise exponential convergence of Sinkhorn iterates and their gradient. Our proof relies on the connection between these iterates and the evolution along the Hamilton-Jacobi-Bellman equations of value functions obtained from SOC -problems. Our approach is novel in that it is purely probabilistic and relies on coupling by reflection techniques for controlled diffusions on the torus.
\end{abstract}

\begin{keywords}%
optimal transport, Sinkhorn algorithm, stochastic optimal control, Schr\"odinger bridge
\end{keywords}

\def\mcx{\mathcal{X}}
\def\msa{\mathsf{A}}
\def\forany{\text{ for any }}

\section{Introduction}\label{sec:intro}

Computational optimal transport (OT) has known great progress over these past few years \citep{peyre2019computational}, and has thus become a popular tool in a wide range of fields such as machine learning \citep{adler2017learning, arjovsky2017wasserstein}, computer vision \citep{dominitz2009texture, solomon2015convolutional}, or signal processing \citep{kolouri2017optimal}. Let $\mu$ and $\nu$ be two probability measures defined on a measurable state space $(\msx,\mcx)$. The primal OT problem \citep{villani2008optimal} between $\mu$ and $\nu$, corresponding to a measurable cost function $\cost\, : \msx^2 \to [0, +\infty)$, can be formulated as solving the optimization problem
\begin{equation}
  \label{ot_pb}
    \inf_{\pi \in \Pi(\mu, \nu)}\int \cost(x,y)\De \pi(x,y) \, ,
\end{equation}
where $\Pi(\mu, \nu)$ is defined as the set of couplings between $\mu$ and $\nu$, \ie, $\pi \in \Pi(\mu,\nu)$ if $\pi(\msa \times \msx) = \mu(\msa)$ and $\pi(\msx\times \msa)= \nu(\msa)$ for any $\msa \in \mcx$. This problem admits the following dual formulation
\begin{equation}
  \label{ot_dual}
    \sup_{(\varphistar,\psistar) \in \mathcal{R}(\cost)}\int \{\varphistar(x)+\psistar(y)\}\De (\mu \otimes \nu)(x,y) \, ,
\end{equation}
where $$\mathcal{R}(\cost)=\{(\varphistar, \psistar)\in \mathcal{C}(\msx)^2: \forany (x,y)\in \msx^2, \varphistar(x) +\psistar(y) \leq \cost(x,y)\}$$ is the set of ``Kantorovitch potentials'' \citep{kellerer1984duality}. In many applications of OT, $\msx\subset\bbRD$ and one chooses the Euclidean quadratic cost $\cost(x,y)=\|x-y\|^2/2$.  Under this setting, Monge-Kantorovich's theorem states that \eqref{ot_pb} admits a unique minimizer $\pi^\star$. In addition, in the case where $\mu$ admits a density with respect to the Lebesgue measure, Brenier's theorem \citep{brenier1991polar} established that this minimizer is also solution to the Monge problem, \ie, there exists a convex function $\psiMonge : \rset^d \to \rset \cup \{\infty\}$ such that $\pi^\star$ is the pushforward of $\mu$ by the application $x \mapsto (x,\TMonge(x))$ with $\TMonge(x) = \nabla \psiMonge(x)$ if $\psiMonge(x) < \infty$ and $\TMonge(x) = 0$ otherwise (referred to as the ``Monge'' map). Moreover, $\psiMonge$ is related to a Kantorovitch potential $\varphistar$ solving \eqref{ot_dual} as $\psiMonge(x)=\normLigne{x}^2/2-\varphistar(x)$.  Unfortunately, OT problems \eqref{ot_pb} and \eqref{ot_dual} suffer from the curse of dimensionality \citep{papadakis2014optimal, niles2022estimation}, which makes impossible to compute $\pi^\star$ or the map $\TMonge$ in high-dimensional settings. Although recent works have been carried out this problem assuming regularity conditions on the domain $\msx$ or on some densities of $\mu$ and $\nu$, if they exist, solving efficiently \eqref{ot_pb} and \eqref{ot_dual} remains an open problem \citep{benamou2014numerical, niles2022estimation, forrow2019statistical}.

To circumvent these computational limits, an approach consists in computing a regularized version of the OT problem \eqref{ot_pb}, which penalizes the entropy of the joint coupling $\pi$:
\begin{align}\label{ot_pb_reg}
    \inf_{\pi \in \Pi(\mu, \nu)}\defEns{\int \cost(x,y)\De \pi(x,y) + \varepsilon \mathrm{KL}(\pi \mid \mu \otimes \nu) } \, ,
\end{align}
where $\mathrm{KL}$ denotes the Kullback-Leibler divergence and $\varepsilon>0$ is a regularization parameter. The entropic regularization notably defines a  convex minimization problem (in contrast to \eqref{ot_pb} in general settings), which admits a unique solution $\pi_\varepsilon^\star$. In addition, under appropriate conditions, $\{\pi_\varepsilon^\star\}_{\varepsilon >0}$ converges to a solution of \eqref{ot_pb}; see e.g., \cite{leonard2012schrodinger}.  The entropic OT problem \eqref{ot_pb_reg} can be tracked back to Schr\"odinger \citep{Schr} and may be casted as a ``static Schr\"odinger problem'' \citep{LeoSch, conforti2019second, carlier2017convergence} given by
\begin{equation}
  \label{ot_pb_reg_1}
    \inf_{\pi \in \Pi(\mu, \nu)}\mathrm{KL}(\pi \mid \rho_{\varepsilon}) \, ,
\end{equation}
where $\rho_{\varepsilon} \in \cP(\msx \times \msx)$ is the \textit{reference} measure defined by $\De \rho_{\varepsilon} (x,y)/\De\{\mu \otimes \nu\}\propto\exp(-\cost(x,y)/\varepsilon)$.
Under some conditions on $\mu$ and $\nu$,  $\pi_\varepsilon^\star$ admits as  density 
\begin{equation}\label{eq:sch_potentials}
\frac{\De \pi_\varepsilon^\star}{\De \rho_{\varepsilon} }(x,y)=\exp[-(\varphi_\varepsilon(x)+\psi_\varepsilon(y))]\, ,
\end{equation}
where $\varphi_\varepsilon \in \mathrm{L}^1(\mu)$ and $\psi_\varepsilon \in \mathrm{L}^1(\nu)$ are called the ``Schr\"odinger potentials''. These potentials are unique up to a trivial additive constant and can be considered as a regularized version of the Kantorovich potentials $\varphistar$ and $\psistar$. Indeed, under similar assumptions as Brenier's theorem, it holds that the (rescaled) Schr\"odinger potentials and their gradients respectively converge to the Kantorovich potentials and their gradients as $\epsilon$ goes to 0 \citep{lagg2022gradient}, hence recovering the Monge map $\psiMonge$.  In contrast to exact OT problems \eqref{ot_pb} and \eqref{ot_dual}, \eqref{eq:sch_potentials} can be solved quickly using the Sinkhorn algorithm \citep{Sinkhorn64, cuturi2013sinkhorn}, and has thus become a popular alternative to the standard OT formulation.
The Sinkhorn algorithm consists in defining sequences $(\varphi_{\varepsilon,n})_{n \in\nset}$ and $(\psi_{\varepsilon,n})_{n \in\nset}$ respectively approximating $\varphi_{\varepsilon}$ and $\psi_{\varepsilon}$, relying that these two functions are fixed points of a particular functional. In this paper, we are interested in the convergence of these two sequences to the ``Schr\"odinger potentials'' $\varphi_{\varepsilon}$ and $\psi_{\varepsilon}$. More precisely, our contributions are as follows.

\paragraph{Contributions.}

We provide a new approach to study the convergence of the Sinkhorn algorithm for the case where the state space $\msx$ is chosen as the $d$-dimensional torus $\mathbb{T}_L^d = \rset^d/(L\zset^d)$, for $L >0$, endowed with its canonical Riemannian metric. In particular, our analysis exploits the relationship between the Schr\"odinger bridge problem and Stochastic Optimal Control (SOC). As shown by \cite{LeoSch} for the case $\msx=\bbRD$, $\pi_\varepsilon^\star$ is the distribution of the pair of random variables $(X_0,X_1)$, where $(X_t)_{t \in[0,1]}$ evolves along the stochastic differential equation $\De X_t=u_\varepsilon^\star(t,X_t) \De t + \sqrt{\varepsilon} \,\De B_t$ and $u_\varepsilon^\star$ is the \textit{control} function solving  
  \bes
  \label{eq:control_intro}
\begin{split}
  \inf_{u: [0,1] \times \bbRD \to \bbRD} \frac12\int_0^1 \mathbb{E}\|u(t,X_t)\|^2 \De t\quad\text{such that}\quad
\begin{cases}\De X_t=u(t,X_t) \De t + \sqrt{\varepsilon} \,\De B_t \eqsp,\\
  X_0\sim \mu \eqsp, X_1\sim \nu \eqsp,
\end{cases}
\end{split}
\ees
where $(B_t)_{t \geq 0}$ is a standard Brownian motion over $\bbRD$.
By establishing new convergence bounds for inhomogeneous controlled processes on $\torusL^d$ related to \eqref{eq:control_intro} using coupling techniques, we show the pointwise exponential convergence of the sequence of Sinkhorn iterates. While this result is not new, our approach is new in that it is essentially probabilistic in nature.
More importantly, this approach allows us to prove our second contribution, namely the convergence of the gradient for the Sinkhorn iterates $(\nabla \varphi_{\varepsilon,n})_{n \in\nset}$ and $(\nabla \psi_{\varepsilon,n})_{n \in\nset}$ to the gradients of the Schr\"odinger potentials, $\nabla \varphi_{\varepsilon}$ resp. $\nabla \psi_{\varepsilon}$, which are used to estimate the Brenier map.  
To the best of our knowledge, our analysis is the first to derive convergence of gradients independently from iterates' convergence \footnote{During the discussion period before acceptance of the present paper, one of the reviewers pointed out that it could be possible to adapt \cite[Lemma 4.8]{delBarrio2022} to obtain convergence of the sequence of the gradients from the convergence of Sinkhorn iterates. However, the constants that would appear in the resulting convergence bounds are not explicit.}. We highlight that this approach is of primary interest since it can be directly generalized to the unbounded setting.
\paragraph{Outline of the work.} The paper is organized as follows. In Section \ref{sec:ass_and_res}, we introduce the theoretical setting of our analysis of Sinkhorn algorithm, detail our assumptions and present our main result. In Section \ref{sec:explicit_rates}, we discuss the dependence of the convergence rate in the parameters of the problem. We review related work and precisely detail our contributions in Section \ref{sec:related_work}, and present the main steps of our proof in Section \ref{sec:proof}.

\paragraph{Notation.}For any measurable space $(\msx, \mathcal{X})$, we denote by $\cP(\msx)$ the space of probability measures defined on $(\msx, \mathcal{X})$. 
Denote by $\rml^1 (\msx)$ the set of function integrable with respect to $\mu$. For any two distributions $\mu,\nu \in \cP(\msx)$, we define the Kullback--Leibler divergence between $\mu$ and $\nu$ as  $\mathrm{KL}(\mu \mid \nu)= \int_{\msx} \De \mu \log(\De \mu/\De \nu)$ if $\mu \ll \nu$ and $\mathrm{KL}(\mu \mid \nu)= +\infty$ otherwise. In the case $\msx = \rset^d$, we denote by $\Leb$ the Lebesgue measure and  define $\Ent(\mu) = \int \rmd \Leb \log (\rmd \mu / \rmd \Leb) $ if  $\mu \ll \Leb$ and $+\infty$ otherwise. 

\section{Theoretical framework and main results} \label{sec:ass_and_res}

\paragraph{Setting and Sinkhorn iterates.}Throughout this paper, we consider two probability measures $\mu$ and $\nu$ defined on the torus $\mathbb{T}_L^d\coloneqq \bbRD /L\mathbb{Z}^d$ of length $L >0$. Since $\torus_L^d$, endowed with addition, is a compact Lie group \cite[Chapter 15]{bump2013lie}, we denote by $\HaarL^d$ the left Haar measure which corresponds to its Riemannian volume form \cite[Chapter 11.4]{folland2013real}. Furthermore, we consider in our paper the problem \eqref{ot_pb_reg_1} for a particular class of reference measure $\rho_{\varepsilon}$:  for a 
fixed time horizon $T>0$, we aim at solving the static Schr\"odinger problem defined by
\begin{equation}\label{eq:sb_ours}
    \inf_{\pi \in \Pi(\mu, \nu)}\mathrm{KL}(\pi \mid \rmR_{0,T}) \, ,
\end{equation}
where $\rmR_{0,T}$ is a distribution on $\tset^{2d}_L$  related to the Langevin stochastic differential equation (SDE) 
\begin{equation}\label{eq:sde_sb}
    \De X_t = -\nabla V(X_t)\De t + \De B_t \,,
\end{equation}
for a twice continuously differentiable potential function $V: \mathbb{T}_L^d \to \mathbb{R}$. Note that for $V\equiv 0$ the above stochastic dynamics corresponds to Brownian motions (\ie, the one associated to the Laplacian operator). If we further consider as state space $\msx = \rset^d$, then $\bfm$ equals the Lebesgue measure $\Leb$ and  \eqref{eq:sb_ours} is an equivalent formulation of the entropic transport problem
\begin{align}\label{ot_pb_reg_our}
    \inf_{\pi \in \Pi(\mu, \nu)}\defEns{\frac12\int \norm{x-y}^2 \De \pi(x,y) + T\,\mathrm{KL}(\pi \mid \mu \otimes \nu)  } \, .
\end{align}
For the general case $V\not \equiv 0$, we refer to  \cite{garcia2019langevin} for an introduction to Langevin diffusion on $\torusL^d$.
Since $V$ is twice continuously differentiable and $\torusL^d$ is compact, by \cite[Theorem 10.1]{kent1978time}, \eqref{eq:sde_sb} admits a unique solution and define a Markov semigroup $(\rmP_t)_{t \geq 0}$ with bi-continuous transition density $(p_t)_{t > 0}$  with respect to the stationary distribution, $\bfm(\rmd x) = \rme^{-2V(x)} \Haar(\rmd x)$, of \eqref{eq:sde_sb} that is symmetric, \ie, for any $x,y$, 
$p_t(x,y) = p_t(y,x)$. As a result, for any $T >0$, $(x,y) \mapsto  \rme^{-2V(x)} p_T(x,y)$ defines a joint density on $(\torusL^{d})^2$ and $\rmR_{0,T}$ is the corresponding probability measure.

We now state our main assumption. In particular, we will suppose that the two distributions $\mu$ and $\nu$ are equivalent to  $\bfm$.

\begin{assumption}
  \label{ass:density_mu_nu}
The potential  $V$ is twice continuously differentiable and there exists two continuously differentiable functions from $\torusL^d$ to $\rset$, $U_{\mu}$ and $U_{\nu}$, such that 
\begin{equation}\label{eq:log:densities}
\mu(\De x) = \exp(-U_{\mu}(x))\mm(\De x) \eqsp, \quad \nu(\De x) = \exp(-U_{\nu}(x))\mm(\De x) \eqsp.
\end{equation}
\end{assumption}

Under Assumption~\ref{ass:density_mu_nu}, $\mathrm{KL}(\mu\mid\mm)$ and $\mathrm{KL}(\nu \mid\mm)$ are finite and  \cite[Theorem 2.6]{LeoSch} shows that Problem \eqref{eq:sb_ours} admits a unique minimizer $\pistar\in \cP(\mathbb{T}_L^d \times \mathbb{T}_L^d)$ dominated by $\rmR_{0,T}$, which can be expressed via Schr\"odinger potentials $\varphistar, \psistar: \torusL^d \to \rset \cup \{\infty\}$ such that
\begin{equation}\label{plan:potentials}
  \frac{\De \pistar}{\De \rmR_{0,T}} (x,y)=\exp(-\varphistar(x)-\psistar(y))\,.
\end{equation}
Since $p_T$ is continuously differentiable with respect to its both variables \citep{kent1978time}, \citep[Lemma 4.11]{nutz2021introduction} implies that $\varphistar$ and $\psistar$ are also continuous and even Lipschitz. In fact, we will recover this result as a corollary of our results. 

  Here, we assume the potentials $\varphistar,\psistar$ satisfying the symmetric normalization $
\int\varphistar\De\mu+\mathrm{KL}(\mu\mid\mm)=\int\psistar\De\nu+\mathrm{KL}(\nu \mid\mm)$.
  Then, the  Sinkhorn algorithm \citep{Sinkhorn64,marinogerolin2020} consists in defining  the sequence of potentials $(\varphi_n)_{n \in \mathbb{N}}$ and $(\psi_n)_{n \in \mathbb{N}}$, starting from $\psi^0=0$\footnote{Let us point out the fact that our results hold true for any smooth choice of $\psi^0$. Here, we set $\psi^0=0$ for convenience.}, by the recursion: for $n\in \mathbb{N}$
  \begin{equation}
    \label{eq:sinkhorn_iterate}
\varphi^{n+1}\coloneqq U_{\mu}+ \log \rmP_T \rme^{-\psi^n}
 \eqsp, \qquad \psi^{n+1} \coloneqq U_{\nu}+\log \rmP_T \rme^{-\varphi^{n+1}} \eqsp.
\end{equation}
where $(\rmP_t)_{t\geq 0}$ is the semigroup associated to the SDE \eqref{eq:sde_sb}. From \eqref{eq:log:densities} and \eqref{plan:potentials}, it is immediate to deduce that the couple $(\varphistar, \psistar)$ is a fixed point of the above iteration. Moreover, the algorithm can be interpreted as fixing one of the prescribed marginals at each step. More precisely, when $\psi^n$ is given and we compute the next iterate $\varphi^{n+1}$, we are implicitly prescribing the couple $(\varphi^{n+1},\psi^n)$ to fit the first marginal constraint, \ie, we are imposing that the first marginal of the probability measure $\De \pi^{n+1,n}/\De \rmR_{0,T}\propto\exp(-\varphi^{n+1}(x)-\psi^n(y))$ is exactly $\mu$. At the next iteration, when we compute $\psi^{n+1}$ we forget about the first marginal and impose the constraint on the second one, which yields to imposing the second marginal of $\De \pi^{n+1,n+1}/\De \rmR_{0,T}\propto\exp(-\varphi^{n+1}(x)-\psi^{n+1}(y))$ to be equal to $\nu$. On the primal side this is also equivalent to minimizing at each step the $\mathrm{KL}$-divergence from the previous plan subject to a one-sided marginal constraint, \ie,
\begin{equation}
\label{eq:primal_pb}
\pi^{n+1,n}\coloneqq{\arg\min}_{\Pi(\mu,\star)} \mathrm{KL}(\cdot|\pi^{n,n})\eqsp, \qquad\pi^{n+1,n+1}\coloneqq {\arg\min}_{\Pi(\star,\nu)} \mathrm{KL}(\cdot|\pi^{n+1,n})\eqsp,
\end{equation}
where $\Pi(\mu,\star)$ (resp. $\Pi(\star,\nu)$) is the set of probability measures on $(\torusL^d)^2$ such that the first marginal is $\mu$ (resp. the second marginal is $\nu$). Let us also point out that the choice of an optimal-enough regularization parameter $T$, which guarantees both fast convergence of Sinkhorn algorithm and accurate approximation of OT, is still a very active field of research. For instance, on a discrete setting (with $n$-atomic supports), \cite{altschuler2017} suggests that choosing $T=\log(n)/\tau$ is enough in order to get a $\tau$-accuracy with just $O(\log(n)/\tau^3)$ iterations. We refer to \cite{peyre2019computational} for a further discussion on this trade-off.

In order to be consistent with the normalization imposed on the Schr\"odinger potentials $\varphistar$ and $\psistar$, we might have to normalize at each step the obtained iterates by considering for any $n\in \mathbb{N}$ 
 \bes
 \varphi^{\diamond n}=\varphi^n-\biggl(\int\varphi^n\De\mu-\int\varphistar\De\mu\biggr)\eqsp, \qquad 
 \psi^{\diamond n}=\psi^n-\biggl(\int\psi^n\De\nu-\int\psistar\De\nu\biggr)\,,
 \ees
 so that
 \begin{equation}
   \label{eq:equality_int}
\text{   $\int\varphi^{\diamond n}\De\mu=\int\varphistar\De\mu$ \, and \,  $\int\psi^{\diamond n}\De\nu=\int\psistar\De\nu $} \eqsp.
 \end{equation}
One may also consider other normalization options, such as  the pointwise condition $\varphistar(0)=\psistar(0)=0$, or the zero-mean normalization \citep{marinogerolin2020, carlier2020differential, deligiannidis2021quantitative, Carlier22multisink}

We consider on the torus $\mathbb{T}_L^d$ a \emph{sine-distance} which suits best our periodic situation. More precisely, for any pair $(x,\,y)\in\mathbb{T}_L^d\times \mathbb{T}_L^d$, we define 
\begin{equation}\label{def:delta}
\delta(x,y)= L\,\sqrt{\sum_{i=1}^d \sin^2\biggl(\frac{\pi}{L}(x^i-y^i)\biggr)}\in [0,L\eqsp d^{1/2}] \, ,
\end{equation}
where $x=(x^i)_{i\in[d]}$, $y=(y^i)_{i\in[d]}$ and the difference $(\pi/L) (x^i-y^i)$ has to be thought as an element of the one dimensional unit-torus $\mathbb{T}^1=\mathbb{S}^1$ identified with the unit-circle. Note that the above sine-distance is indeed a distance (the triangular inequality follows from the properties of $\sin$) and is equivalent to the flat-distance $\sfd$ induced by the Euclidean distance function:
\[(\pi\eqsp L)^{1/2}\,\sfd(x,y)\leq \delta(x,y)\leq\,\pi\,\sfd(x,y)\,.\] Let us remark here that our motivation behind adapting such equivalent metric comes from the coupling techniques considered in Appendix \ref{app:gio}, where we need to consider a smooth metric on~$\torusL^d$.

Finally, we define the Lispchitz norm of a function $h:\mathbb{T}_L^d \to \mathbb{R}$ as
\bes
\|h\|_{\mathrm{Lip}}\coloneqq \sup_{x\neq y\in\mathbb{T}_L^d}\frac{|h(x)-h(y)|}{\sfd(x,y)} \, .
\ees
We are now ready to state our main result. 

\begin{theorem}\label{thm:sink}
Assume Assumption~\ref{ass:density_mu_nu}. Then, there exist a rate $\gamma\in(0,1)$ and a positive constant $\cconst>0$ such that
\begin{equation}\label{eq:exp:conv}
\begin{aligned}
\sup_{x\in\mathbb{T}_L^d} |\varphi^{\diamond n}(x)-\varphistar(x)|\leq &\, Ld^{1/2}  \cconst \, \gamma^{2\,n-1}\, \|\psi^{ 0}-\psistar\|_{\mathrm{Lip}}\\
\sup_{x\in\mathbb{T}_L^d} |\psi^{\diamond n}(x)-\psistar(x)|
\leq&\, Ld^{1/2}  \cconst \, \gamma^{2\,n}\,\|\psi^{0}-\psistar\|_{\mathrm{Lip}}\,.
\end{aligned}
\end{equation}
Similarly, we get the uniform exponential convergence for the gradients
\begin{equation}\label{eq:exp:conv:grad}
\begin{aligned}
\sup_{x\in\mathbb{T}_L^d} |\nabla\varphi^{\diamond n}(x)-\nabla\varphistar(x)|\leq & \pi \cconst\, \gamma^{2\,n-1}\,\|\psi^{0}-\psistar\|_{\mathrm{Lip}}\\
\sup_{x\in\mathbb{T}_L^d} |\nabla\psi^{\diamond n}(x)-\nabla\psistar(x)|
\leq& \pi \cconst\, \gamma^{2\,n}\,\|\psi^{0}-\psistar\|_{\mathrm{Lip}}\,.
\end{aligned}
\end{equation}
Moreover, $\gamma$ and $\cconst$ have an explicit expression that can be computed, depending on the choice of the potential $V$, see \eqref{eq:explicit_constants}. 
\end{theorem}

As detailed in the proof of Theorem~\ref{thm:sink} given in \Cref{sec:proof}, for any potential $V$ satisfying Assumption~\ref{ass:density_mu_nu}, there exists an explicit rate $\blambdaV>0$ such that the rate $\gamma$, given by Theorem \ref{thm:sink}, can be written as  $\gamma=\rme^{-\blambdaV\pi^2 T}$. In fact, $\blambdaV$ corresponds to the ergodicity rate of the controlled diffusion when considering as underlying reference system the diffusion driven by $b_s(x)=-\nabla V(x)+\nabla\log \rmP_{T-s}e^{-\psistar}(x)$, \ie, the Schr\"odinger Bridge SDE \citep{follmer1997gantert}.

\section{Explicit convergence rates and discussion}
\label{sec:explicit_rates}

In this section, we provide explicit \textit{estimates} of $\gamma$ and $\cconst$, defined in Theorem \ref{thm:sink}, for a potential $V$ which is assumed to be $\alpha$-semiconvex for some $\alpha\leq 0$\footnote{Let us point out that we cannot expect $\alpha >0$ since we work on a compact Riemannian manifold.}, \ie, $V$ is satisfies for any $x,y \in \torusL^d$,
\begin{equation}\label{def:sin:a:conv}
\sin\biggl(\frac{\pi}{L}(x-y)\biggr)^{{\sf  T}}(\nabla V(x)-\nabla V(y))\geq\frac{\pi\,\alpha}{2L}\,\delta(x,y)^2  \, ,
\end{equation}
where the $\sin$ function applied to any vector of $\torusL^d$ as to be understood as a component-wise map applied to a representative in $[-\pi/2,+\pi/2)$.
An example of such potential is provided in Appendix~\ref{subsec:example_V}. 
For notations' sake, let us denote with $D=L\,d^{1/2}$ the diameter of the torus $\mathbb{T}_L^d$ and let $\eta_D=\exp(D^2\, |\alpha|/8)$.
Then the estimates of $\gamma$ and $\cconst$ are given by 
\bes
\label{eq:estimate_gamma}
\log\gamma\leq\, -\pi^2 T\frac{|\alpha|/4}{\eta_D-1}\,\exp\left(- D\frac{ \|U_{\mu}\|_{f_V}\vee\|U_{\nu}\|_{f_V}}{1-\exp\left(-\frac{ |\alpha|/4 }{\eta_D-1}\,\pi^2\,T\right)}\right)
\ees
and
\bes
\label{eq:estimate_c}
     \cconst\leq 2\frac{\eta_D}{\sqrt{L \pi}}\,\exp\left(D\frac{ \|U_{\mu}\|_{f_V}\vee\|U_{\nu}\|_{f_V}}{1-\exp\left(-\frac{|\alpha|/4  }{\eta_D-1}\, \pi^2\,T\right)}\right)\,,
\ees
where $f_V:  \rset_+ \to \rset_+$ is a concave and continuous function, employed in our proofs in order to get exponential contractive Lipschitz estimates, whereas 
the $f_V$-Lipschitz norm $\|\cdot\|_{f_V}$ is defined for any function $\phi : \torusL^d \to \rset$ as 
\bes 
\| \phi\|_{f_V} \coloneqq \sup_{x\neq y\in\mathbb{T}_L^d} \frac{|\phi(x)-\phi(y)|}{f_V(\delta(x,y))}\,.
\ees
The proof of these bounds is postponed to Appendix~\ref{app:explicit:computations}. 

Moreover, when considering $V=0$ we recover the classic setting when considering Brownian motions, which corresponds to the quadratic regularized OT, and the computations in Appendix \ref{app:explicit:computations} show that the rate of convergence $\gamma_0$ in the asymptotic regime $T\to 0$ behaves as
\begin{equation}\label{eq:asympt:gamma0}
\log\gamma_0\sim -\pi^2\,D_{\mu,\nu}^2\,{D^4}\,T^{-1}\,\exp(- D_{\mu,\nu}\,D^3\,T^{-1})\quad\text{ and }\quad
\cconst\sim\exp(D_{\mu,\nu}\,D^3\,T^{-1}) \eqsp.
\end{equation}
where $D_{\mu,\nu}\coloneqq \frac{1}{2\pi^2}\,\|U_{\mu}\|_{f_0}\vee\|U_{\nu}\|_{f_0}$.

The general bounds \eqref{eq:estimate_gamma} and \eqref{eq:estimate_c} (for $\alpha<0$), in the asymptotic regime  $T \to 0$ and $D \to + \infty$,  may be reduced to
\begin{align}
  |\log \gamma|&=\mathcal{O}\left(T \,\eta_D^{-1} \exp(- \eta_D \, D \,T^{-1})\right)\, ,\\
  \cconst&=\mathcal{O}(\eta_D \exp(\eta_D\, D \, T^{-1}))\, ,
\end{align}
where we omitted the constants that do not affect significantly this regime. As expected, $\gamma \to 1$ and $\cconst \to \infty$, \ie, the asymptotic regime highly slows down the convergence of Sinkhorn algorithm, especially when $d$ is large.

\section{Comparison with existing literature and original contributions}
  \label{sec:related_work}

The  Sinkhorn algorithm is very well-known and its study has been intensified particularly after \cite{cuturi2013sinkhorn}. Nonetheless, its introduction dates back to \cite{Yule12}, and it is often referred to as Iterative Proportional Fitting Procedure (IPFP).
We refer to  \cite{peyre2019computational} for an extensive overview on Entropic Optimal Transport, on Sinkhorn algorithm, on its generalizations and on their applications

 On discrete state spaces, convergence of the Sinkhorn algorithm has been proven for the first time by \cite{Sinkhorn64} and \cite{SinkhornKnopp67}. In this setting, \cite{Franklin89hilbert} show that Sinkhorn algorithm is equivalent to a sequence of iterations of a contraction in the Hilbert projective metric and prove its geometric (\ie, exponential) convergence by relying on Birkoff's theorem. We refer also to \cite{borwein1994entropy}, who focus on fixed-point problems, in settings more general than the matrix one. In particular, they consider Sinkhorn-type algorithms which turn out to be once again equivalent to iterations of a contraction in the Hilbert projective metric.
  
 The continuous counterpart of the Hilbert metric has already been investigated in \cite{chen2016hilbertmetric} and in \cite{deligiannidis2021quantitative}. In the latter, the authors provide quantitative stability estimates of IPFP on compact metric spaces, from which its convergence can be deduced. Even though these original approaches provide also quantitative rates of convergence, they badly scale when applied to a multimarginal Optimal Transport setting. Recently, new ideas from convex theory have been introduced in order to tackle the convergence of Sinkhorn algorithm in the multimarginal setting too, for bounded costs (or equivalently compact spaces). Along this line of research, it is worth mentioning \cite{carlier2020differential} where the authors prove well-posedness of Sinkhorn iterates and their smooth dependence from the marginals. In addition,  \cite{marinogerolin2020} have proven an $\mrl^p$ (qualitative) convergence of Sinkhorn iterates, and \cite{Carlier22multisink} which improves the previous results by showing an exponential convergence (with a rate that scales linearly with the number of marginals).   

Regarding the primal formulation and the convergence of $(\pi^{n,n})_{n \in\nset}$ defined in \eqref{eq:primal_pb} to the optimal coupling, \cite{NutzWiesel:qualitative:Sinkhorn} establish  qualitative convergence in total variation. Following this work, \cite{eckstein2021quantitative} show quantitative (polynomial) convergence in Wasserstein distance.  Lastly, \cite{ghosal2022nutz} derive  polynomial linear convergence (\ie, of order $O(1/n)$) with respect to a symmetric relative entropy.

\bigskip
\noindent\textbf{Original Contribution.} In this paper we provide a new approach in the study of Sinkhorn algorithm on the $d$-dimensional torus $\mathbb{T}_L^d$ and its main novelties can be summarized as follows.
\begin{itemize}
\item Our proofs rely on probabilistic arguments and coupling methods, by exploiting the connection between the Schr\"odinger potentials and value functions of stochastic optimal control problems. To the best of our knowledge, this is the first paper addressing the problem relying on a (non-trivial) stochastic interpretation, while the existing literature usually relies on convex analysis and/or on the Hilbert metric. Moreover, this probabilistic approach via stochastic optimal control could in principle be carried over to the unbounded case (e.g. in the Euclidean space $\bbRD$).
Here we specify our results on the torus since it allows us to work on a compact state space while benefiting from its underlying Euclidean structure. However
 our approach could be extended to smooth compact manifolds without boundaries but at the expense of technicalities, in particular the definition of an appropriate coupling by reflection. Dealing with the torus allows us to reduce these complications at the bare minimum 
avoid technical details related to general bounded compact state spaces for which topological conditions generally have to be imposed.

\item We prove the convergence of Sinkhorn iterates as a corollary of the convergence of their gradients (or equivalently in Lipschitz norms). To the best of our knowledge, our result is the first one addressing the problem directly at the level of the gradients.
Moreover, our probabilistic approach provides Lipschitz estimates 
along solutions of Hamilton-Jacobi-Bellman equations (\ie, for any time $s\in[0,T]$; see \eqref{lip:bound}). 

Our results should be compared to   \cite{deligiannidis2021quantitative} where Lipschitz estimates close to ours are given,  but for iterates $(f_n,g_n)_{n \in\nset}$ corresponding to $f_n = \rmP_T\,\rme^{-\psi^n}$ and $g_n=\rmP_T\,\rme^{-\varphi^n}$. To show their result,  \cite{deligiannidis2021quantitative} rely on Birkoff's theorem for the Hilbert metric since the iterations they consider are then just linear updates. Note that convergence of the gradient of $(\varphi_n,\psi_n)_{n\in\nset}$ cannot be deduce from their result since $\varphi_n$ and $\psi_h$ are non-linear transformation of $g_n$ and $f_n$ respectively.  

\item We get an exponential rate of convergence $\gamma=\rme^{-\blambdaV\,\pi^2\,T}$ which converges to $1$ as $T\downarrow 0$. This exponential dependence on $T$ is not surprising. Indeed it is well known that convergence of Sinkhorn algorithm implies quantitative stability (continuous) bounds for Schr\"odinger problem (and entropically regularized optimal transport, see \cite{eckstein2021quantitative}) while on the contrary the optimal transport map is solely $1/2$-H\"older continuous by \cite{gigli2011holder}.
\end{itemize}

  \section{Sketch of the proof}\label{sec:proof}

  We now introduce the main components of our method to analyse the convergence of the Sinkhorn iterates given by \eqref{eq:sinkhorn_iterate}.
  We first introduce the function $\{\mathcal{U}^{T,h}_t\}_{t \in [0,T]}$ defined for any measurable and bounded function $h: \torusL^d \to \rset$:
  \begin{equation}
    \label{eq:def_value_function_exp}
    \mathcal{U}^{T,h}_t = - \log \rmP_{T-t} \rme^{-h} \eqsp,
  \end{equation}
which can be shown, by a direct computation, to correspond to the solution of the Hamilton-Jacobi-Bellman (hereafter HJB) equation defined by
\begin{equation}\label{HJB}
\begin{cases}
\partial_t u_t+\frac12\Delta u_t-\nabla V\cdot\nabla u_t-\frac12 |\nabla u_t|^2=0\\
u_T=h\, .
\end{cases}
\end{equation}
With these notations,
Sinkhorn iterates can be written as
\begin{equation}
    \label{eq:sinkhorn_iterates}
\varphi^{n+1}= U_{\mu}-\mathcal{U}^{T,\psi^n}_0\eqsp, \qquad \psi^{n+1} = U_{\nu}-\mathcal{U}^{T,\varphi^{n+1}}_0 \eqsp.
\end{equation}
To get some bounds on the Lipschitz constant of $\varphi^{n+1}$ and $\psi^{n+1}$, we then show that if $h: \torusL^d \to \rset$ is Lipschitz, $\mathcal{U}^{T,h}_0$ is also Lipschitz with an explicit bound for its Lipschitz constant.

To do so, we use that $\mathcal{U}^{T,h}_t$ can be represented also as the value function of the SOC problem
\begin{equation}
  \label{eq:soc_main}
\begin{split}
  \mathcal{U}^{T,h}_t(x)= \inf_{q\in \mathcal{A}_{[t,T]}} \bbE\biggl[\frac{1}{2}\int_t^T|q_s|^2\De s +h(X^q_T)\biggr]\,
  \\
  \text{ where }\,
\begin{cases}\De X^q_s=(-\nabla V(X^q_s)+q_s)\De s + \De B_s\\
X^q_t=x\,,\end{cases}
\end{split}
\end{equation}
where $(B_s)_{s \geq 0}$ is a $(\mcf_s)_{s \geq 0}$-Brownian motion on $\torusL^d$  defined on the filtered probability space $(\Omega,\PP,\mcf,(\mcf_s)_{s \geq 0})$ satisfying the usual conditions, and $\mathcal{A}_{[t,T]}$ denotes the set of admissible controls, \ie, $(\mcf_s)_{s \geq 0}$-progressively measurable processes.  We provide a precise statement of this result in Proposition \ref{app:prop:contr} in Appendix \ref{app:gio}, 
where we show the optimal control process to be the feedback-process $q_s=-\nabla\mathcal{U}_s^{T,h}(X_s^q)$. Moreover, Proposition \ref{app:prop:contr} in Appendix \ref{app:gio} provides a non-trivial control of $\|\mathcal{U}^{T,h}_t\|_{\mathrm{Lip}}$ for any function $h\in \mathcal{C}^3(\mathbb{T}_L^d)$. We give here the main ideas of the proof of this result.

We first show that for any pair of stochastic processes $(X_s,Y_s)_{s \in [t,T]}$, starting from $X_t=x$ and $Y_t=y$ respectively, solution of
\begin{align}
  \label{eq:2}
  \De X_s&=-\nabla V (X_s)\De s-\nabla \mathcal{U}^{T,h}_s(X_s)\De s+\De B_s\eqsp, \\
  \De Y_s&= - \nabla V(Y_s) \rmd s -\nabla \mathcal{U}^{T,h}_s(X_s)\De s+\De \tilde{B}_s \eqsp,
\end{align}
where $(\tilde{B}_s)_{s \geq 0}$ is a $(\mcf_s)_{s \geq 0}$-Brownian motion on $\torusL^d$ defined on $(\Omega,\mcf,\PP)$, it holds by \eqref{eq:soc_main},
 \begin{equation}
 \mathcal{U}^{T,h}_t(y)-\mathcal{U}^{T,h}_t(x)\leq
 \mathbb{E}\biggl[h(Y_T)-h(X_T)\biggr] \eqsp.
\end{equation}
Then, if $h$ is Lipschitz, we consider a particular coupling which is adapted from the usual  coupling by reflection for homogeneous diffusion \citep{wang1994neumann,eberle2016reflection}  to bound $ \mathcal{U}^{T,h}_t(y)-\mathcal{U}^{T,h}_t(x)$. In particular, the novelty of our approach relies in employing coupling by reflection techniques for controlled diffusion processes on the torus, endowed with the distance $\delta$ given in \eqref{def:delta}, that defines a smooth distance on $\mathbb{T}_L^d$ which is equivalent to the Riemannian distance $\sfd$. An adaptation of the coupling by reflection techniques under this sine-distance is given in Appendix \ref{app:gio}. Owing to the construction given there,
we obtain for any $t\in[0,T]$ and any $h : \torus_L^d \to \rset$, Lipschitz 
\begin{equation}
\| \mathcal{U}^{T,h}_t\|_{f_V}\leq \rme^{-\lambda_V \,\pi^2\,(T-t)}\|h\|_{f_V}
\label{eq:lip_U_main}
\end{equation}
where the rate $\lambda_V >0$ and the function $f_V:  \rset_+ \to \rset_+$, which is concave and continuous, are defined in \eqref{def:f} and \eqref{def:lambda:C}. 
Moreover, we prove in Proposition \ref{app:costruzione:f} in Appendix \ref{app:gio} that $f_V$ is equivalent to the identity. Therefore, $\| \cdot \|_{f_V}$ is equivalent to the usual Lipschitz norm $ \|\cdot \|_{\mathrm{Lip}}$ (\ie, with $f_V$ being the identity and considering the flat-distance) since $\delta(\cdot,\cdot)$ is equivalent to $\sfd(\cdot,\cdot)$. In particular, we have for any function $\phi : \torusL^d \to \rset$
\begin{equation}
\frac{1}{\pi}\,\|\phi\|_{\mathrm{Lip}}\leq \|\phi\|_{f_V}\leq \frac{C_V^{-1}}{\sqrt{L\eqsp \pi}}\,\|\phi\|_{\mathrm{Lip}}\eqsp.
\label{eq:equiv_lip_norm}
\end{equation}
where $C_V$ is defined in \eqref{def:lambda:C}. By combining \eqref{eq:lip_U_main} with \eqref{eq:equiv_lip_norm}, we are then able to bound $\|\mathcal{U}^{T,h}_t\|_{\mathrm{Lip}}$ as
\begin{equation}\label{eq:contr:type}
\| \mathcal{U}^{T,h}_t\|_{\mathrm{Lip}}\leq C_V^{-1}\sqrt{\frac{\pi}{L}}\,\rme^{-\lambda_V \,\pi^2\,(T-t)}\|h\|_{\mathrm{Lip}}
\end{equation}

It is then possible studying how the Lipschitzianity propagates along Sinkhorn iterates using the following result.

\begin{lemma}\label{lemma:fnorm_bound}
  Assume Assumption~\ref{ass:density_mu_nu}. 
For all $n\geq 0$ we have
\be\label{eq:fnorm_bound1}
\begin{split}
\|\varphi^{n+1}\|_{f_V} \leq \|U_{\mu} \|_{f_V}+   \rme^{-\lambda_V\,\pi^2\, T} \|\psi^n\|_{f_V}\\
\|\psi^{n+1}\|_{f_V} \leq \|U_{\nu} \|_{f_V}+   \rme^{-\lambda_V \,\pi^2\,T} \|\varphi^{n+1}\|_{f_V}\,
\end{split}
\ee
Moreover, for all $n\geq 1$ we have
\be\label{eq:fnorm_bound2}
\begin{split}
\|\psi^n \|_{f_V} \leq \frac{\|U_{\nu}\|_{f_V}+\exp(-\lambda_V \,\pi^2\,T)\|U_{\mu}\|_{f_V}}{1-\exp(-2\lambda_V\,\pi^2\, T)} \\
\|\varphi^n \|_{f_V}\leq  \frac{\|U_{\mu}\|_{f_V}+\exp(-\lambda_V\,\pi^2\, T)\|U_{\nu}\|_{f_V}}{1-\exp(-2\lambda_V \,\pi^2\,T)},
\end{split}
\ee
\end{lemma}
\begin{proof}
As shown in Proposition \ref{app:prop:contr} in Appendix \ref{app:gio} (see also \eqref{eq:lip_U_main}), the Lipschitz-regularity backward propagates along solutions of HJB equations. Particularly, it holds 
\bes
\| \mathcal{U}^{T,\psi^n}_0\|_{f_V}\leq \|\psi^n\|_{f_V}\, \rme^{-\lambda_V\,\pi^2\, T} \, ,
\ees
which, combined with \eqref{eq:sinkhorn_iterates} and an application of the triangular inequality, gives the first claim in \eqref{eq:fnorm_bound1}. The second claim follows by symmetry. Concatenating these two bounds, we obtain
\bes
\|\psi^{n+1}\|_{f_V} \leq \|U_{\nu} \|_{f_V}+   \rme^{-\lambda_V \,\pi^2\,T}\|U_{\mu}\|_{f_V}+\rme^{-2\lambda_V \,\pi^2\,T}\|\psi^{n} \|_{f_V},
 \ees
from which the first relation in \eqref{eq:fnorm_bound2} follows by induction. The second relation follows by symmetry.
\end{proof}

 \begin{remark}
Let us also point out that the Lipschitz estimates obtained in Lemma \ref{lemma:fnorm_bound}, as well as the ones proven in Lemma \ref{lemma:diff:lip} below, hold true also for the normalized Sinkhorn iterates $\varphi^{\diamond n},\psi^{\diamond n}$ (and for any other trivial additive perturbation of them). Indeed any additive normalization would cancel out when considering Lipschitz norms.
\end{remark}

From the pointwise convergence of the normalized Sinkhorn iterates $\varphi^{\diamond n},\psi^{\diamond n}$ towards the Schr\"odinger potentials (which in our compact and smooth setting is guaranteed from the geometric $\rml^p$ convergence in \cite{marinogerolin2020}), the previous regularity result propagates to the potentials and using \eqref{eq:equiv_lip_norm},  we obtain the following corollary.
\begin{corollary}
  Assume Assumption~\ref{ass:density_mu_nu}. Then it holds 
\be\label{eq:fnorm_bound3}
\begin{split}
\|\psistar\|_{f_V} \leq \frac{\|U_{\nu}\|_{f_V}+\exp(-\lambda_V\,\pi^2\, T)\|U_{\mu}\|_{f_V}}{1-\exp(-2\lambda_V \,\pi^2\,T)} \\
\|\varphistar\|_{f_V}\leq  \frac{\|U_{\mu}\|_{f_V}+\exp(-\lambda_V \,\pi^2\,T)\|U_{\nu}\|_{f_V}}{1-\exp(-2\lambda_V \,\pi^2\,T)}
\end{split}
\ee
and therefore the Lipschitz norm of the Schr\"odinger potentials can be bounded as
\bes
\begin{split}
\|\psistar\|_{\mathrm{Lip}} \leq \pi\,\frac{\|U_{\nu}\|_{f_V}+\exp(-\lambda_V \,\pi^2\,T)\|U_{\mu}\|_{f_V}}{1-\exp(-2\lambda_V\,\pi^2\, T)} \\
\|\varphistar\|_{\mathrm{Lip}}\leq  \pi\,\frac{\|U_{\mu}\|_{f_V}+\exp(-\lambda_V \,\pi^2\,T)\|U_{\nu}\|_{f_V}}{1-\exp(-2\lambda_V\,\pi^2\, T)}\,.
\end{split}
\ees
\end{corollary}

We are now ready to prove the key contraction estimates, from which our main result follows. Once again the main idea behind our proof is relying on a stochastic control problem where the Schr\"odinger potential contributes in the final cost while its gradient drives the controlled SDE. This allows to back-propagate along an HJB equation the Lipschitz regularity of the difference between the Sinkhorn iterates and the target Schr\"odinger potential. Indeed, if we denote with $\mathcal{D}_t^n\coloneqq \mathcal{U}_t^{T,\psi^n}-\mathcal{U}_t^{T,\psistar}$ (the difference between the evolutions along HJB of $\psi^n$ and $\psistar$ respectively) from \eqref{HJB} we deduce that it solves
\begin{equation}\label{pde:diff}
\begin{cases}
\partial_t u_t+\frac12\Delta u_t+(-\nabla V-\nabla \mathcal{U}^{T,\psistar}_t)\cdot\nabla u_t-\frac12 |\nabla u_t|^2=0\\
u_T=\psi^n-\psistar\,,
\end{cases}
\end{equation}
 which can be represented (see Proposition \ref{app:prop:contr} in Appendix \ref{app:gio}) as the value function of the stochastic control problem
\be\label{stoch:diff}
\begin{split}
\mathcal{D}^n_t(x)=&\, \inf_{q_\cdot} \bbE\biggl[\frac{1}{2}\int_t^T|q_s|^2\De s +\psi^n(X^q_T)-\psistar(X^p_T)\biggr]\\
\quad\text{where}&\quad
\begin{cases}\De X^q_s=(-\nabla V(X^q_s)-\nabla\mathcal{U}^{T,\psistar}_s (X^q_s)+q_s)\De s + \De B_s\\
X^q_t=x\,.\end{cases}
\end{split}
\ee
Once the connection with the stochastic optimal control formulation is established, the proof boils down once again in studying how Lipschitz-regularity backward propagates along solutions of HJB equations. 

\begin{lemma}\label{lemma:diff:lip}
  Assume Assumption~\ref{ass:density_mu_nu}. 
There exist $\blambdaV >0$, given by \eqref{eq:lambda_bar} in Appendix~\ref{app:proof_lemm}, and a continuous concave function $\bar f_V :\rset_+ \to \rset_+$ such that 
\be\label{eq:grad_exp_2}
\begin{split}
    \|\psi^{n+1}-\psistar\|_{\bar f_V}\leq  \exp(-\blambdaV \,\pi^2\,T)\| \varphi^{n+1}-\varphistar\|_{\bar f_V}\\
\|\varphi^{n+1}-\varphistar\|_{\bar f_V}\leq  \exp(-\blambdaV\,\pi^2\, T)\| \psi^{n}-\psistar\|_{\bar f_V}\,.
\end{split}
\ee
As a result
\be\label{eq:grad_exp_3}
\begin{split}
\|\psi^{n+1}-\psistar\|_{\bar f_V}\leq  \exp(-2\blambdaV \,\pi^2\,T)\| \psi^{n}-\psistar\|_{\bar f_V}\\
\|\varphi^{n+1}-\varphistar\|_{\bar f_V}\leq  \exp(-2\blambdaV \,\pi^2\,T)\| \varphi^{n}-\varphistar\|_{\bar f_V} \eqsp.
\end{split}
\ee
\end{lemma}
\begin{proof}
  The proof is postponed to Appendix~\ref{app:proof_lemm}.
\end{proof}

We can now complete the proof of Theorem \ref{thm:sink}.\\

\begin{proof}[Proof of Theorem \ref{thm:sink}.]\\
Since $\int\varphi^{\diamond n}\De\mu=\int\varphistar\De\nu$ (see \eqref{eq:equality_int}), uniformly on $x \in\mathbb{T}_L^d$, it holds
\begin{align}
|\varphi^{\diamond n}(x)-\varphistar(x)|&=\biggl|\varphi^{\diamond n}(x)-\int_{\mathbb{T}_L^d}\varphi^{\diamond n}\,\De\mu-\varphistar(x)+\int_{\mathbb{T}_L^d}\varphistar\,\De\mu\biggr|\\
&= \biggl|\int_{\mathbb{T}_L^d}\biggl[(\varphi^n-\varphistar)(x)-(\varphi^n-\varphistar)(y)\biggr]\De \mu(y)\biggr| \\
&
\leq\int_{\mathbb{T}_L^d}\biggl|(\varphi^n-\varphistar)(x)-(\varphi^n-\varphistar)(y)\biggr|\De \mu(y)\\
&\leq \|\varphi^n-\varphistar\|_{\bar f_V} \int_{\mathbb{T}_L^d} \bar f_V(\delta(x,y))\,\De\mu(y)\\
&\leq L\eqsp d^{1/2}\,\|\varphi^n-\varphistar\|_{\bar f_V}\,.
\end{align}
Therefore, by concatenating \eqref{eq:grad_exp_3} along $n$ iterates, we end up with
\bes
\begin{split}
\sup_{x\in\mathbb{T}_L^d} |\varphi^{\diamond\,n+1}(x)-\varphistar(x)|\leq &\,L \, d^{1/2} \,\exp(-2\,n\,\blambdaV \,\pi^2\,T)\|\varphi^{1}(x)-\varphistar(x)\|_{\bar f_V}\\
\leq&\, L \, d^{1/2} \,\exp(-(2\,n+1)\,\blambdaV \,\pi^2\,T)\|\psi^{0}(x)-\psistar(x)\|_{\bar f_V} \eqsp .
\end{split}
\ees
By reasoning in the same fashion, since $\int\psi^{\diamond n}\De\mu=\int\psistar\De\nu$ or simply by relying on \eqref{eq:fnorm_bound2}, we conclude that
\bes 
\sup_{x\in\mathbb{T}_L^d} |\psi^{\diamond n}(x)-\psistar(x)|
\leq L \, d^{1/2} \,\exp(-2\,n\,\blambdaV \,\pi^2\,T)\|\psi^{0}(x)-\psistar(x)\|_{\bar f_V}\,.
\ees
Using \eqref{eq:equiv_lip_norm_f_bar}, we may conclude the proof of \eqref{eq:exp:conv} by setting
\begin{equation}\label{eq:explicit_constants}
    \gamma = \exp(-\blambdaV \,\pi^2\,T) \text{ and } \cconst= \bar{C}_V^{-1}/\sqrt{L \, \pi} \eqsp ,
\end{equation}
where  $\blambdaV$ and $\bar{C}_V$ are respectively defined at \eqref{eq:lambda_bar} and \eqref{eq:equiv_lip_norm_f_bar}.

The proof of the convergence of the gradients can be obtained in a similar fashion since
\bes
\sup_{x\in\mathbb{T}_L^d}|\nabla\varphi^{\diamond n}-\nabla\varphistar|(x)=\sup_{x\in\mathbb{T}_L^d}|\nabla\varphi^{ n}-\nabla\varphistar|(x)\leq\|\varphi^{ n}-\varphistar\|_{\mathrm{Lip}}\leq \pi\|\varphi^{ n}-\varphistar\|_{\bar f_V}
\ees
and similarly for $\psi^{\diamond n}-\psistar$,
from which \eqref{eq:exp:conv:grad} follows by concatenating the contraction in \eqref{eq:grad_exp_3}.
\end{proof}

\section{Conclusion}\label{sec:conclu}
In this paper, we have introduced a new probabilistic approach in the study of Sinkhorn algorithm. We have shown that each iteration is equivalent to solving an Hamilton-Jacobi-Bellman equation, \ie, computing the value function of a stochastic control problem, and showed that the Lipschitz regularity of the previous Sinkhorn iterate propagates to the next one, with a constant dissipative rate. From this contraction estimates we have deduced the exponential convergence of the Sinkhorn iterates and of their gradients. All the dissipative Lipschitz estimates for the value function of the stochastic control problems considered have been deduced via an application of coupling by reflection techniques for controlled diffusion on the torus.

 This approach is a complete novelty and could in principle be extended to the non-compact Euclidean case, problem that we 
 address in the follow-up work \citep{conforti2023Sinkhorn}.

\acks{GG thanks École Polytechnique for its hospitality, where this research has been carried out, and NDNS+ for funding his visit there
(NDNS/2023.004 PhD Travel Grant). GG is supported from the NWO Research Project 613.009.111 "Analysis meets Stochastics:
Scaling limits in complex systems". 

\noindent GC acknowledges funding from the grant SPOT (ANR-20-CE40-0014).

\noindent AD acknowledges support from the Lagrange
Mathematics and Computing Research Center. AD would like
to thank the Isaac Newton Institute for Mathematical Sciences for support and hospitality during the programm
\emph{The mathematical and statistical foundation of future data-driven engineering} when work on this paper was undertaken. This work was supported by: EPSRC grant number
EP/R014604/1.}

\bibliography{bibliography}

\newpage
\appendix

\section*{Organisation of the supplementary}

In Appendix~\ref{app:gio}, we give an explicit construction for coupling by reflection on the torus, which we adapt for controlled drifts. In Appendix~\ref{app:proof_lemm}, we provide the proof of Lemma~\ref{lemma:diff:lip}, which is crucial to derive our main result. In Appendix~\ref{app:explicit:computations}, we first compute explicit estimates for the convergence rate $\gamma$ and the constant $\cconst$ appearing in Theorem \ref{thm:sink}, under a weak-semiconvexity assumption, and then give a class of examples of potentials satisfying this condition.

\section{Reflection coupling for HJB estimates on the torus}\label{app:gio}
In this section we adapt the ideas developed in \cite{conforti2022coupling} for controlled diffusions on $\bbRD$ to our compact setting on $\mathbb{T}_L^d$ as \eqref{eq:2}. Before doing so, let us spend a few sentences on why we cannot rely on the existing coupling by reflection literature. First of all, the SDE that we consider is time-inhomogenoeus which is not usually studied, especially when considering diffusions on a Riemannian manifold such as the torus. Moreover being on the torus brings another issue. As long as $Y_s$ does not belong to the cut-locus of $X_s$, we may straightforwardly define coupling by reflection as in the flat Euclidean case.
However as soon as $Y_t$ belongs to the cut-locus of $X_t$ a more careful analysis is required since we could not apply Ito formula to $\sfd^2$ anymore. Indeed on any Riemannian manifold  $\msm$, $\sfd^2$ fails to be $\mathcal{C}^2$ on the set of conjugates points, \ie, couples $(x,y)\in \msm\times \msm$ such that $y$ is in the cut-locus of $x$ \cite[Chapter 5.9.1]{petersen2006Riemannian}. In order to deal with this issue one could try to restrict the problem to regular domains (submanifolds without conjugate points and with convex boundary) as it is done in \cite{wang1994neumann}, or one could try to combined reflection coupling with different techniques. The latter approach is portrayed in \cite{cranston91}, where the author alternates coupling by reflection with an independent coupling (in a small-time window, enough to avoid that the coupled diffusion reaches points in the cut-locus). However we could neither rely on this approach because, despite of having chosen independent Brownian motions, in our control setting the process $(Y_t)_{t \geq 0}$ would still depend on $(X_t)_{t \geq 0}$ via the control process (see \eqref{sde:system} where we are going to define the reflection coupling by plugging the optimal control for $(X_t)_{t \geq 0}$ in the controlled SDE for $(Y_t)_{t \geq 0}$).

\bigskip

Having in mind our target, we begin by adapting the coupling by reflection technique of \cite{wang1994neumann,eberle2016reflection} to the torus. In this appendix, we are going to consider a time-inhomogeneous SDE on the torus, namely
\begin{equation}
  \label{eq:inhom_diffusion}
\De X_s=b_s(X_s)\,\De s+\De B_s\quad \text{ on }\mathbb{T}_L^d\,
\end{equation}

with a time-inhomogeneous drift function $b_s\in \mathcal{C}^{0,1}([0,T)\times\mathbb{T}^d_L)$.
Under the condition $b_s\in \mathcal{C}^{0,1}([0,T)\times\mathbb{T}^d_L)$, this SDE admits a unique strong solution and corresponds to a inhomogeneous Markov semigroup that we denote by $(\rmP_{s,t}^b)_{s,t \in [0,T]}$ in the sequel. 
We also consider an arbitrary modulus of weak-semiconvexity $\kappa_b$ associated to $b$, whose definition is provided below. 

\begin{definition}\label{app:new:def:kappa} A function $\kappa_b: (0,L\,d^{1/2}] \to \mathbb{R}$ is said to be a modulus of weak semi-convexity associated to the drift $b$ if (i) $\kappa_b$ is continuous on $(0, L\,d^{1/2}]$, (ii) $s \mapsto s \kappa_b(s)$ is integrable on $(0, L \, d^{1/2}]$ and (iii) we have
\bes
    \kappa_b(r)\leq\inf_{s\in[0,T]}\inf\biggl\{-\frac{2L}{\pi}\frac{\sin(\frac{\pi}{L}(x-y))^{{\sf  T}}(b_s(x)-b_s(y))}{\delta^2(x,y)}\,:\,x\neq y\in\mathbb{T}_L^d\text{ s.t. }\delta(x,y)=r \biggr\}\,.
\ees
Let us remark that $\kappa_b$ is always non-positive \footnote{This property is inherited from the fact that we work on a compact Riemannian manifold.}.
\end{definition}
 By mimicking \citep{wang1994neumann,eberle2016reflection}, let us define for any $r\in(0,L \, d^{1/2}]$ the functions 
\begin{equation}
\begin{split}\label{def:f}
f_b(r)\coloneqq&\, \int_0^r\phi(s)\,g(s)\,\De s\quad\text{where}\quad
\phi(r)\coloneqq \exp\left(\frac14\int_0^r s\,\kappa_b(s)\,\De s\right)\,,\\\Phi(r)\coloneqq &\,\int_0^r\phi(s)\,\De s\quad\text{and}\quad
g(r)\coloneqq\, 1-\frac{\int_0^{r}\Phi(s)/\phi(s)\,\De s}{2\int_0^{L\,d^{1/2}}\Phi(s)/\phi(s)\,\De s}\,.
\end{split}
\end{equation}
Let us remark here that the main difference with \citep{eberle2016reflection} is the presence of $L\,d^{1/2}$, \ie the diameter of the torus, as an upper-limit in the integration domain in the definition of $g$. 
Finally, consider the multiplicative and rate constants
\begin{equation}\label{def:lambda:C}
\begin{split}
&\,C_b\coloneqq \frac{\phi( L \, d^{1/2})}{2}=\frac{1}{2}\,\exp\left(\frac14\int_0^{L\,d^{1/2}} s\,\kappa_b(s)\,\De s\right)\\
&\,\text{ and }\quad\lambda_b\coloneqq \left(\int_0^{L \, d^{1/2}}\Phi(s)/\phi(s)\,\De s\right)^{-1}.
\end{split}
\end{equation}
The key properties of the triplet $(C_b,\lambda_b,f_b)$ are summarized in the following result.
\begin{proposition}\label{app:costruzione:f}
The triplet $(C_b,\lambda_b,f_b)$ satisfies
\begin{enumerate}
\item $f_b$ is equivalent to the identity, \ie, for any $ r\in[0,L \, d^{1/2}]$,
\bes C_b\,r\leq f_b(r)\leq r\quad\text{and}\quad C_b\leq f_b^{'}(r)\leq 1\quad \, \ees
\item for any $r\in[0,L\,d^{1/2}]$, it holds
\begin{equation}\label{eq:diff:f}
f^{''}_b(r)-\frac{\kappa_b(r)}{4}\,r\,f^{'}_b(r)\leq -\frac{\lambda_b}{2}f_b(r)\,.
\end{equation}
\end{enumerate}
\end{proposition}
\begin{proof}
Let us start by simply noticing that the non-positivity of $\kappa_b$ implies
\bes
\phi(L\eqsp d^{1/2})\leq \phi(s)\leq 1\,,\quad\phi(L\eqsp d^{1/2})\,r\leq \Phi(r)\leq r\quad\text{ and }\quad \frac12\leq g(r)\leq 1\,,
\ees
which immediately proves the bounds for $f_b^{'}(r)=\phi(r)g(r)$ with $C_b=\phi(L\,d^{1/2})/2$. From the previous bound on $g$ we immediately deduce also
\begin{equation}\label{eq:bound:Phi:f}
    \Phi(r)/2\leq f_b(r)\leq \Phi(r)\,,
\end{equation}
which combined with the above bound for $\Phi$ concludes the proof of the first item.

\medskip

In order to prove the second item, it is enough to compute
\bes 
\begin{split}
f_b^{''}(r)=\phi^{'}(r)g(r)+\phi(r)g^{'}(r)=\frac{\kappa_b(r)}{4}\,r\,\phi(r)g(r)+\phi(r)g^{'}(r)
\end{split}
\ees
which combined with  $f_b^{'}(r)=\phi(r) g(r)$, reads as
\bes
f^{''}_b(r)-\frac{\kappa_b(r)}{4}\,r\,f^{'}_b(r)\leq \phi(r)g^{'}(r)\,.
\ees
Since, for any $r\in[0,L\,d^{1/2}]$, it holds $g^{'}(r)= -\frac{\lambda_b}{2}\,\Phi(r)/\phi(r)$, we deduce  
\bes
f^{''}_b(r)-\frac{\kappa_b(r)}{4}\,r\,f^{'}_b(r)\leq -\frac{\lambda_b}{2}\,\Phi(r)\overset{\eqref{eq:bound:Phi:f}}{\leq} -\frac{\lambda_b}{2}\,f_b(r)\,.
\ees
\end{proof}

We are now in place to apply the coupling by reflection ideas to the HJB equation.
Recall that we consider a filtered probability space
$(\Omega, (\mathcal{F}_s)_{s\in[0,T]},\mathcal{F}, \mathbb{P})$ satisfying the usual conditions and endowed with a Brownian motion on $\torusL^d$.

Recall that for any $h : \torusL^d \to \rset$, Lipschitz, its backward evolution along the semigroup associated with \eqref{eq:inhom_diffusion} is defined as
\[\mathcal{V}^{T,h}_t\coloneqq-\log \rmP_{t,T}^b \rme^{-h}\quad\forall t\in[0,T]\,.\]
Let us recall that if the SDE is time-homogeneous (e.g., $b_s(x)=V(x)$ for all $s\in [0,T]$) then the previous expression is equivalent to the usual expression $-\log \rmP_{T-t} \rme^{-h}$.
As already mentioned in Section \ref{sec:intro}, this quantity can be seen as the value function of a stochastic optimal control problem, and solves an Hamilton-Jacobi-Bellman equation. We prove this statement in the following result.
\begin{proposition}\label{app:prop:contr} We recall that $b_s\in\mathcal{C}^{0,1}([0,T)\times\mathbb{T}_L^d)$. Let $h\in \mathcal{C}^3(\mathbb{T}_L^d)$, then 
\begin{enumerate}
\item \label{item:SOC_1}$\mathcal{V}^{T,h}_t$ is the unique strong solution in $\mathcal{C}^{1,2}([0,T)\times\mathbb{T}_L^d)$ of the HJB equation
\begin{equation}\label{app:HJB}
\begin{cases}
\partial_t u_t+\frac12\Delta u_t+b_t(X_t)\cdot\nabla u_t-\frac12 |\nabla u_t|^2=0\\
u_T=h\,.
\end{cases}
\end{equation}
\item \label{item:SOC_2}$\mathcal{V}^{T,h}_t$ is the the value function of the stochastic control problem
\be\label{app:control:problem}
\begin{split}
\mathcal{J}^{T,h}_t(x)=&\, \inf_{q_\cdot\in\mathcal{A}_{[t,T]}} \bbE\biggl[\frac{1}{2}\int_t^T|q_s|^2\De s +h(X^q_T)\biggr]\\
&\text{where }\mathbb{P}\text{-a.s. it holds }\,
\begin{cases}\De X^q_s=(b_s(X^q_s)+q_s)\De s + \De B_s\\
X^q_t=x\end{cases}
\end{split}
\ee
and $\mathcal{A}_{[t,T]}$ denotes the set of admissible controls, \ie, progressively measurable processes with finite moments on $(\Omega, (\mathcal{F}_s)_{s\in[0,T]},\mathcal{F}, \mathbb{P})$. Moreover, the optimal control is a feedback-process equal to $-\nabla\mathcal{V}_s^{T,h}(X^q_s)$.
\item \label{item:SOC_3}  
Uniformly on $t\in[0,T]$, it holds
\bes
\| \mathcal{V}^{T,h}_t\|_{f_b}\leq \rme^{-\lambda_b \,\pi^2\,(T-t)}\|h\|_{f_b}
\ees
where $\lambda_{b}, f_{b}$ are the quantities defined at \eqref{def:f} and \eqref{def:lambda:C}.
\end{enumerate}
Moreover, Item~\ref{item:SOC_3} can be relaxed to assuming $h$ only Lipschitz.
\end{proposition}
\begin{proof}
Let $h\in\mathcal{C}^3(\mathbb{T}_L^d)$. Firstly, let us observe that, under the current assumption, a direct computation shows that $\mathcal{V}^{T,h}_t$ is a classical solution for the HJB associated the generator $L_t=\frac{\Delta}{2}+b_t\cdot\nabla$ and that
\eqref{app:HJB} can be equivalently written as
\bes
\begin{cases}
\partial_t u_t+\frac12\Delta u_t+H(x,\nabla u_t)=0\\
u_T=h\,;
\end{cases}\quad\text{where }H(x,p)\coloneqq -\inf_{u\in\bbRD}\left\{\frac{|u|^2}{2}+b_s(x)\cdot p+\eqsp u\cdot p\right\} \eqsp,
\ees
where the above infimum is actually attained at $\omega(x,p)\coloneqq -p$. The uniqueness of the classical solution stated in Item~\ref{item:SOC_1} can be deduced via the uniqueness for the HJB equation in the Euclidean space \cite[Proposition 3.1]{conforti2022coupling}, by meaning of a periodic extension argument (\ie, considering the same HJB where the coefficients are defined on $\bbRD$ via periodicity).

\medskip

Item~\ref{item:SOC_2} can be proven
 via a standard approximation procedure,  that we sketch here for readers' convenience. 
Let us start by introducing for any $M\in\mathbb{N}$ the Hamiltonian $H^M\colon \mathbb{T}^d_L\times \bbRD\to \mathbb{R}$
\bes
H^M(x,p)\coloneqq -\min_{|u|\leq M}\left\{\frac{|u|^2}{2}+b_s(x)\cdot p+\chi_M(|u|)\eqsp u\cdot p\right\}\eqsp,
\ees
where $\chi^M$ is a smooth function satisfying
\bes
\chi^M(r)=\begin{cases}
    1\quad\text{if }r\leq M\\
    0\quad\text{if }r\geq M+1
\end{cases}\quad\text{ such that }\quad\sup_{r\geq 0}\left|\chi_M^{'}\right|<+\infty\eqsp.
\ees
Denote with $\omega^M(x,p)$ the optimal $u$ in the definition of $H^M$, \ie, the solution of
\bes
u+\chi_M(|u|)p+|u|\eqsp\chi_M^{'}(|u|)\eqsp p=0\eqsp,
\ees
and notice that for any $|p|< M$, it holds $\omega^M(x,p)=-p$.

Then, for any fixed $M\in\mathbb{N}$, \cite[Theorem 3.1]{zhu2011} guarantees the uniqueness of  a viscosity solution for the associated HJB equation
\begin{equation}\label{app:hjb:troncata}
\begin{cases}
\partial_t u_t+\frac12\Delta u_t+H^M(x,\nabla u_t(x))=0\\
u_T=h\,,
\end{cases}
\end{equation}
which is equal to the value function $u^M$ of the corresponding stochastic optimal control problem and that it is jointly continuous (in time and space) (cf. \cite[Theorem 3.7 and Remark 2.7]{zhu2010}). 

Next, we claim that
\begin{equation}\label{app:eq:appo:appo}
\sup_{x\in\mathbb{T}^d_L\eqsp
t\in[0,T]
}\left|\nabla \mathcal{V}^{T,h}_t(x)\right|\leq K\eqsp,
\end{equation}
for some positive constant $K$. By recalling that $\mathcal{V}_t^{T,h}\coloneqq -\log \rmP_{t,T}^b \rme^{-h}$, with $(\rmP^b_{t,s})_{0\leq t\leq s\leq T}$ being the semigroup associated to \eqref{eq:inhom_diffusion}, we may write
\bes
\mathcal{V}_t^{T,h}(x)=-\log \mathbb{E}[\exp(-h(X^{t,x}_T))]\quad\text{ where }\begin{cases}
\De X_s^{t,x}=b_s(X_s^{t,x})\De s+\De B_s\eqsp,\\
X_t^{t,x}=x\eqsp.
\end{cases}
\ees
Fix $y\in\mathbb{T}^d_L$ and consider a synchronous coupling with the previous process, \ie, write
\bes
\mathcal{V}_t^{T,h}(y)=-\log \mathbb{E}[\exp(-h(X^{t,y}_T))]\quad\text{ where }\begin{cases}
\De X_s^{t,y}=b_s(X_s^{t,y})\De s+\De B_s\eqsp,\\
X_t^{t,y}=y\eqsp,
\end{cases}
\ees
where we have used the same Brownian motion for the two SDEs above. This, the regularity of $h\in\mathcal{C}^3(\mathbb{T}^d_L)$ and the compactness of the torus allow us to write
\bes
\begin{split}
\mathcal{V}_t^{T,h}(x)-\mathcal{V}_t^{T,h}(y)=\log\biggl(1+\frac{\mathbb{E}[\rme^{-h(X^{t,y}_T)}-\rme^{-h(X^{t,x}_T)}]}{\mathbb{E}[\exp(-h(X^{t,x}_T))]}\biggr)\leq K\eqsp\mathbb{E}[\rme^{-h(X^{t,y}_T)}-\rme^{-h(X^{t,x}_T)}]\\
\leq K \mathbb{E}[\sfd(X^{t,x}_T,X^{t,y}_T)]\leq K\sfd(x,y)\,.
\end{split}
\ees
In the above derivation the constant $K$ differs at each step and we have relied on synchronous coupling technique in order to obtain a closed-form Gronwall-type estimate for $\sfd(X^{t,x}_T,X^{t,y}_T)$. We omit this explicit derivation since we are going to perform similar computations for the coupling by reflection later, and we are not interested in the magnitude of $K$. Let us just remark that $K$ depends on the size of the torus, on the Lipschitz norm of the drift and of $h$ and it can be taken uniformly with respect to $t\in[0,T]$.
This concludes the proof of the bound claimed in \eqref{app:eq:appo:appo}.

Since, we already know that $\mathcal{V}_t^{T,h}$ is a classical solution of \eqref{app:HJB}, the uniform estimate \eqref{app:eq:appo:appo} implies that $\mathcal{V}_t^{T,h}$ is also a classical solution of the truncated \eqref{app:hjb:troncata} for any $M\geq K$. Since we have uniqueness of viscosity solutions for \eqref{app:hjb:troncata}, we finally conclude that  
for any $M\geq K$, our function $\mathcal{V}_t^{T,h}$ is the unique viscosity solution, \ie the value function of the stochastic control problem associated to $H^M$. By letting $M\to+\infty$, we conclude the proof of Item~\ref{item:SOC_2}.

\medskip

We now show Item~\ref{item:SOC_3}. Let us start by considering on $(\Omega, (\mathcal{F}_s)_{s\in[0,T]},\mathcal{F}, \mathbb{P})$ the toroidal SDEs
\begin{equation}\label{sde:system}
\begin{cases}
\De X_s=b_s(X_s)\De s-\nabla \mathcal{V}^{T,h}_s(X_s)\De s+\De B_s\quad\forall s\in[t,T] \\
\De Y_s=b_s(Y_s)\De s-\nabla \mathcal{V}^{T,h}_s(X_s)\De s+\De \hat{B}_s\quad\forall s\in[t,\tau]\quad\text{and }Y_s=X_s\quad\forall s\in[\tau,T]\\
(X_t,Y_t)=(x,y)\in\mathbb{T}_L^d\times\mathbb{T}_L^d
\end{cases}
\end{equation}
where 
\bes
\tau\coloneqq \inf\{s\in[t,T]:X_s=Y_s\}\wedge T\,,
\ees
and the \emph{reflected} Brownian motion $\hat{B}$ is defined as
\begin{equation}\label{def:refl:brownian}
\De \hat{B}_s\coloneqq (\operatorname{I}-2\,e_s\,e_s^{{\sf  T}}\,\mathbf{1}_{\{s< \tau\}})\, \De B_s\qquad\text{where}\quad e_s\coloneqq \frac{\sin(\frac{\pi}{L}(X_s-Y_s))}{\|\sin(\frac{\pi}{L}(X_s-Y_s))\|}\,
\end{equation}
where the $\sin$ function applied to any vector of $\torusL^d$ as to be understood as a component-wise map applied to a representative in $[-\pi/2,+\pi/2)$. In the sequel, we consider the same extension for the $\cos$ function. 
By L\'evy's characterization, $\hat{B}$ is an $\mathcal{F}_s$-adapted Brownian motion on $\mathbb{T}_L^d$. The existence and well-posedness of the above coupling process is well known in literature ; see e.g. \cite{chen1989coupling}, where the authors build such coupling by considering a martingale problem associated to a well-chosen elliptic operator.

Therefore from the sub-optimality of $-\nabla \mathcal{V}^{T,h}_s(X_s)_{s\in[t,T]}$ as a control process in the stochastic optimal control problem over the probability space $\Sigma=(\Omega,(\mathcal{F}_s)_{s\in[t,T]},\mathbb{P},(\hat{B}_s)_{s\in[t,T]})$, we deduce that
\bes
\label{eq:proof_hjb_1}
 \mathcal{V}^{T,h}_t(y)-\mathcal{V}^{T,h}_t(x)\leq
 \mathbb{E}\biggl[h(Y_T)-h(X_T)\biggr]\leq \|h\|_{f_b} \mathbb{E}[f_b(\delta(X_T,Y_T))]\,.
 \ees

 Next we show that $\rme^{\lambda\, s}f_b(\delta (X_s,Y_s))$ is a super-martingale. For notations' sake in what follows we introduce the shortcut $\pi_L\coloneqq \pi/L$.
From a first application of Ito formula for any $i\in[d]$ and for any $s\in[t,\tau)$ we have
\bes
\begin{split}
 \De(\sin^2({\pi_L}& (X^i_s-Y^i_s))= {\pi_L}\sin(2{\pi_L}(X^i_s-Y^i_s))(b^i_s(X_s)-b^i_s(Y_s))\De s\\
&\qquad + 4{\pi_L}^2 (\rme^i_s)^2\cos(2{\pi_L}(X_s^i-Y_s^i))\De s
+2{\pi_L} \,\sin(2{\pi_L}(X_s^i-Y_s^i))\,\sum_{k=1}^d(e_s\,e_s^{{\sf  T}})_{ik}\,\De B^k_s\\
& \qquad = {\pi_L}\sin(2{\pi_L}(X^i_s-Y^i_s))(b^i_s(X_s)-b^i_s(Y_s))\De s\\
& \qquad \qquad +  4{\pi_L}^2 \frac{\sin^2({\pi_L}(X_s^i-Y_s^i))}{\|\sin({\pi_L}(X_s-Y_s))\|^2}\cos(2{\pi_L}(X_s^i-Y_s^i))\De s \\
& \qquad \qquad + 2{\pi_L} \,\sin(2{\pi_L}(X_s^i-Y_s^i))\,\sum_{k=1}^d(e_s\,e_s^{{\sf  T}})_{ik}\,\De B^k_s\,.
\end{split}\ees
By summing up  we get
\bes
\begin{split}
\De(\|\sin({\pi_L}(X_s-&Y_s))\|^2)  = {\pi_L}\sin(2{\pi_L}(X_s-Y_s))^{{\sf  T}}(b_s(X_s)-b_s(Y_s))\De s\\
&  \qquad + \frac{4{\pi_L}^2}{\|\sin({\pi_L}(X_s-Y_s))\|^2}\sum_{i=1}^d\sin^2({\pi_L}(X_s^i-Y_s^i))\cos(2{\pi_L}(X_s^i-Y_s^i)) \De s \\
&  \qquad + 2{\pi_L} \,\sin(2{\pi_L}(X_s-Y_s))^{{\sf  T}}\,e_s\,e_s^{{\sf  T}}\,\De B_s\,,
\end{split}\ees
from which we can deduce using $\cos(2\theta) = \cos^2(\theta) - \sin^2(\theta)$,
\bes
\begin{split}
\De\delta(X_s,Y_s)  =L\eqsp\|\sin({\pi_L}(X_s-Y_s))\|^{-1}\biggl\{\frac{{\pi_L}}{2}\sin(2{\pi_L}(X_s-Y_s))^{{\sf  T}}(b_s(X_s)-b_s(Y_s))\De s\\
+\frac{2{\pi_L}^2}{\|\sin({\pi_L}(X_s-Y_s))\|^2}\sum_{i=1}^d\sin^2({\pi_L}(X_s^i-Y_s^i))\cos(2{\pi_L}(X_s^i-Y_s^i))\De s\\
+{\pi_L}\sin(2{\pi_L}(X_s-Y_s))^{{\sf  T}}\,e_s\,e_s^{{\sf  T}}\,\De B_s\biggr\}\\
-\frac{L\,{\pi_L}^2}{2} \|\sin({\pi_L}(X_s-Y_s))\|^{-5}\langle \sin(2{\pi_L}(X_s-Y_s)),\sin({\pi_L}(X_s-Y_s))\rangle^2\De s
\end{split}
\ees

\bes
\begin{split}
\phantom{\De\delta(X_s,Y_s)  } =L\eqsp\|\sin({\pi_L}(X_s-Y_s))\|^{-1}\biggl\{\frac{{\pi_L}}{2}\sin(2{\pi_L}(X_s-Y_s))^{{\sf  T}}(b_s(X_s)-b_s(Y_s))\De s\\
+\frac{{\pi_L}^2}{2}\frac{\|\sin(2{\pi_L}(X_s-Y_s))\|^2}{\|\sin({\pi_L}(X_s-Y_s))\|^2}\De s- \frac{2{\pi_L}^2}{\|\sin({\pi_L}(X_s-Y_s))\|^2}\sum_{i=1}^d\sin^4({\pi_L}(X_s^i-Y_s^i))\De s\\
+{\pi_L}\sin(2{\pi_L}(X_s-Y_s))^{{\sf  T}}\,e_s\,e_s^{{\sf  T}}\,\De B_s\\
-\frac{{\pi_L}^2}{2\,\|\sin({\pi_L}(X_s-Y_s))\|^{4}}\langle \sin(2{\pi_L}(X_s-Y_s)),\sin({\pi_L}(X_s-Y_s))\rangle^2\De s\biggr\}
\end{split}
\ees
or equivalently, by setting $r_s\coloneqq \delta(X_s,Y_s)$ for notations' sake

\bes
\begin{split}
    \De r_s=r_s^{-1} \biggl\{\frac{{\pi \, L}}{2}\sin(2{\pi_L}(X_s-Y_s))^{{\sf  T}}(b_s(X_s)-b_s(Y_s))\De s+{\pi \, L}\sin(2{\pi_L}(X_s-Y_s))^{{\sf  T}}\,e_s\,e_s^{{\sf  T}}\,\De B_s\\
+\frac{(\pi \, L)^2}{2\,r_s^2}\|\sin(2{\pi_L}(X_s-Y_s))\|^2\De s- \frac{2(\pi \, L)^2}{r_s^2}\sum_{i=1}^d\sin^4({\pi_L}(X_s^i-Y_s^i))\De s\\
-\frac{\pi^2 L^4}{2\,r_s^{4}}\langle \sin(2{\pi_L}(X_s-Y_s)),\sin({\pi_L}(X_s-Y_s))\rangle^2\De s\biggr\}\,.
\end{split}
\ees
By applying Ito formula with $f_b$ (defined at \eqref{def:f}) we deduce 
\bes
\begin{split}
    & \De f_b(r_s) =\frac{f_b^{'}(r_s)}{r_s}\biggl\{
    \frac{{\pi \, L }}{2}\sin(2{\pi_L}(X_s-Y_s))^{{\sf  T}}(b_s(X_s)-b_s(Y_s))\De s \\
    & \quad +{\pi \, L}\sin(2{\pi_L}(X_s-Y_s))^{{\sf  T}}\,e_s\,e_s^{{\sf  T}}\,\De B_s 
    +\frac{(\pi \, L)^2}{2\,r_s^2}\|\sin(2{\pi_L}(X_s-Y_s))\|^2\De s \\
     & \quad- \frac{2(\pi \, L)^2}{r_s^2}\sum_{i=1}^d\sin^4({\pi_L}(X_s^i-Y_s^i))\De s
-\frac{\pi^2 \, L^4}{2\,r_s^{4}}\langle \sin(2{\pi_L}(X_s-Y_s)),\sin({\pi_L}(X_s-Y_s))\rangle^2\De s\biggr\}\\
     & \quad+\frac{\pi^2 \, L^4}{2\,r_s^4}f_b^{''}(r_s)\langle \sin(2{\pi_L}(X_s-Y_s)),\sin({\pi_L}(X_s-Y_s))\rangle^2\De s\,.
\end{split}
\ees
Now, notice that
\bes
\begin{split}
\frac{\pi \, L}{2}&\sin(2{\pi_L}(X_s-Y_s))^{{\sf  T}}(b_s(X_s)-b_s(Y_s))\\
=&\pi \, L \sum_{i=1}^d \cos({\pi_L}(X_s^i-Y_s^i))\sin({\pi_L}(X^i_s-Y^i_s))(b^i_s(X_s)-b^i_s(Y_s))\\
\leq& \pi \, L \sum_{i=1}^d \sin(\pi_L(X^i_s-Y^i_s))(b^i_s(X_s)-b^i_s(Y_s))= \pi \, L \sin\biggl(\frac{\pi}{L}(X_s-Y_s)\biggr)^{{\sf  T}}(b_s(X_s)-b_s(Y_s))\\
\leq& -\frac{{\pi^2}}{2} \kappa_b(r_s)\,r^2_s\,,
\end{split}
\ees
where we have from Definition \ref{app:new:def:kappa},
\[\kappa_b(r)\leq \inf_{s\in[0,T]}\inf\biggl\{-\frac{2L}{\pi}\frac{\sin(\frac{\pi}{L}(x-y))^{{\sf  T}}(b_s(x)-b_s(y))}{\delta^2(x,y)}\,:\,x\neq y\in\mathbb{T}_L^d\text{ s.t. }\delta(x,y)=r \biggr\}\,.\]

Therefore we obtain
\bes
\begin{split}
    & \De f_b(r_s)\leq \frac{\pi^2 \, L^4}{2\,r_s^4}f_b^{''}(r_s)\langle \sin(2{\pi_L}(X_s-Y_s)),\sin({\pi_L}(X_s-Y_s))\rangle^2\De s \\ +&\,\frac{f_b^{'}(r_s)}{r_s}\biggl\{
    -\frac{{\pi}^2}{2}\kappa_b(r_s)\,r_s^2\De s+{\pi \, L}\sin(2{\pi_L}(X_s-Y_s))^{{\sf  T}}\,e_s\,e_s^{{\sf  T}}\,\De B_s\\
+&\,\frac{(\pi \, L)^2}{2\,r_s^2}\|\sin(2{\pi_L}(X_s-Y_s))\|^2\De s
-\frac{\pi^2 \, L^4 }{2\,r_s^{4}}\langle \sin(2{\pi_L}(X_s-Y_s)),\sin({\pi_L}(X_s-Y_s))\rangle^2\De s\\
- &\,\frac{2(\pi \, L)^2}{r_s^2}\sum_{i=1}^d\sin^4({\pi_L}(X_s^i-Y_s^i))\De s
\biggr\}   \,.
\end{split}
\ees
Firstly, notice that the quadratic variation term may be trivially bounded since
\bes
\begin{split}
& \frac{L^4}{4}\langle \sin(2{\pi_L}(X_s-Y_s)),\sin({\pi_L}(X_s-Y_s))\rangle^2\\
&\qquad =\biggl(L^2\sum_{i=1}^d\sin^2({\pi_L}(X_s^i-Y_s^i))\cos({\pi_L}(X_s^i-Y_s^i))\biggr)^2\leq r^4_s\,,
\end{split}
\ees
Next, we claim that the summation of  final three terms in the above curly bracket is non-positive. Having in mind that, let us start by noticing that for any $\theta=(\theta^i)_{i\in[d]}\in [0,\pi/2)^d$ it holds\footnote{Let us recall that in the definition of $\sin(\pi_L(x-y))$ we always chose the representative such that $\pi_L(x-y)\in[0\pi/2)$, therefore it is enough considering $\theta\in[0,\pi/2)^d$.} 
\begin{align}
\|&\sin(\theta)\|^2\,\|\sin(2\theta)\|^2
-\langle \sin(2\theta),\sin(\theta)\rangle^2
- 4 \|\sin(\theta)\|^2\, \sum_{i=1}^d\sin^4(\theta^i)\\
&  = \sum_{i,j\in[d]} \sin^2(\theta^i) \sin^2( 2 \theta^j) - \biggl(\sum_{i=1}^d \sin(\theta^i) \sin( 2 \theta^i)\biggr)^2 -4 \sum_{i,j\in[d]} \sin^2(\theta^i) \sin^4( \theta^j)\\
&  = \frac{1}{2}\sum_{i,j\in[d]}\biggl(\sin(\theta^i) \sin( 2 \theta^j) - \sin(2 \theta^i) \sin(\theta^j)\biggr)^2 -4 \sum_{i,j\in[d]} \sin^2(\theta^i) \sin^4( \theta^j)\\
& = \frac{1}{2}\sum_{i,j\in[d]}\biggl(\sin(\theta^i) \sin( 2 \theta^j) - \sin(2 \theta^i) \sin(\theta^j)\biggr)^2 \\
 &\qquad\qquad\qquad-2 \sum_{i,j\in[d]} \sin^2(\theta^i) \sin^2( \theta^j) \biggl(\sin^2(\theta^i) + \sin^2( \theta^j)\biggr)\\
&  = \sum_{i,j\in[d]}\frac{1}{2}\biggl(\sin(\theta^i) \sin( 2 \theta^j) - \sin(2 \theta^i) \sin(\theta^j)\biggr)^2 -2 \sin^2(\theta^i) \sin^2( \theta^j) \biggl(\sin^2(\theta^i) + \sin^2( \theta^j)\biggr)
\end{align}
We are going to show that each term of the above summation is non-positive. Therefore let $x,y\in [0,\pi/2)$ and, owing to the duplication formula for the sine function, notice that 
\begin{align}
     \frac{1}{2}\biggl(\sin(x) \sin( 2 y) - \sin(2 x) \sin(y)\biggr)^2 -2 \sin^2(x) \sin^2( y) \biggl(\sin^2(x) + \sin^2(y)\biggr)\\
    =2\sin^2(x)\sin^2(y)\biggl[\bigl(\cos(y)-\cos(x)\bigr)^2-\sin^2(x)-\sin^2(y)\biggr]\\
     =4\sin^2(x)\sin^2(y)\biggl[\cos^2(x)+\cos^2(y)-\cos(x)\cos(y)-1\biggr]\eqsp.
\end{align}
Therefore our claim follows once we prove that the two-variable function
\bes
a,b\in[0,1]\mapsto f(a,b)\coloneqq a^2+b^2-ab-1
\ees
is non-positive, or equivalently that for any fixed $b\in[0,1]$
the one-variable function 
\bes
a\in[0,1]\mapsto g_b(a)\coloneqq a^2+b^2-ab-1
\ees
is non-positive. This latter claim is always met since (for any fixed $b\in[0,1]$) the above-defined function $g_b$ achieves its maximum value in $a=b/2$ which is negative and reads as
\bes
g_b(b/2)=\frac{3}{4}\,b^2-1< b^2-1\leq 0\eqsp.
\ees

We have therefore proven that 
\bes
\begin{split}
    \De f_b(r_s)& \leq 2{\pi}^2\biggl\{f_b^{''}(r_s)
    -f_b^{'}(r_s)\,\frac{\kappa_b(r_s)}4\,r_s\biggr\}\De s\\
    & + \qquad \frac{f_b^{'}(r_s)}{r_s}\,{\pi\,L}\,\sin(2{\pi_L}(X_s-Y_s))^{{\sf  T}}\,e_s\,e_s^{{\sf  T}}\,\De B_s \, ,
\end{split}
\ees
which combined with the differential property \eqref{eq:diff:f} for the function $f_b$ and with $r_s=0$ on $[\tau,T]$ and $f_b(0)=0$, reads as
\bes
\De f_b(r_s)\leq -\lambda_b\,\pi^2\, f_b(r_s)\,\De s +\pi\,L\,\frac{f_b^{'}(r_s)}{r_s}\,\sin\biggl(\frac{2\pi}{L}(X_s-Y_s)\biggr)^{{\sf  T}}\,e_s\,e_s^{{\sf  T}}\,\De B_s\qquad\forall s\in[t,T]\,.
\ees
By tacking expectation, from Gronwall Lemma we deduce
 \[\mathbb{E}\bigl[f_b(\delta(X_T,Y_T)\bigr]\leq \rme^{-\lambda_b\,\pi^2\,(T-t)}\delta(x,y)\,.\]

Therefore, by \eqref{eq:proof_hjb_1}, for any $x\neq y\in \mathbb{T}_L^d$ we have shown 
 \[\mathcal{V}^{T,h}_t(y)-\mathcal{V}^{T,h}_t(x)\leq \rme^{-\lambda_b\,\pi^2\,(T-t)}\|h\|_{f_b} \delta(x,y)\,,\]
 which ends our proof for Item~\ref{item:SOC_3}.
 
Finally, it is possible to relax Item~\ref{item:SOC_3} to the case $h\in\mathrm{Lip}(\mathbb{T}_L^d)$. This follows by a standard approximation technique, which is detailed in \cite[Lemma 3.1]{conforti2022coupling} for the Euclidean case and that applies straightforwardly to ours (by meaning of a periodic extension argument).
\end{proof}

\section{Proof of Lemma~\ref{lemma:diff:lip}}
\label{app:proof_lemm}

Let us firstly notice that as a byproduct of Proposition \ref{app:prop:contr} in Appendix \ref{app:gio} and \eqref{eq:fnorm_bound3},  we know that
\begin{align}\label{lip:bound}
  \|\mathcal{U}^{T,\psistar}_s\|_{\mathrm{Lip}}\leq \pi\,\|\mathcal{U}^{T,\psistar}_s\|_{f_V} &\leq \pi\,\rme^{-\lambda_V \,\pi^2\,(T-s)}\|\psistar\|_{f_V}
  \\
   &\leq \frac{\pi\,\rme^{\lambda_V \,\pi^2\,s}}{2}\frac{\|U_{\nu}\|_{f_V}+\exp(-\lambda_V \,\pi^2\,T)\|U_{\mu}\|_{f_V}}{\sinh(\lambda_V\,\pi^2\, T)}\,.
\end{align}
Now, if we denote the new drift as $b_s(x)\coloneqq -\nabla V(x)-\nabla \mathcal{U}^{T,\psistar}_s(x)$, for any $x\neq y\in\mathbb{T}_L^d$ such that $\delta(x,y)=r$, we have using \eqref{lip:bound}
\bes
\begin{split}
-\frac{2L}{\pi}\,&\frac{(b_s(x)-b_s(y))\cdot \sin(\frac{\pi}{L}(x-y))}{\delta^2(x,y)}\\
=&\,\frac{2L}{\pi}\,\frac{(\nabla V(x)-\nabla V(y))\cdot\sin(\frac{\pi}{L}(x-y))}{\delta^2(x,y)}+\frac{2L}{\pi}\,\frac{(\nabla \mathcal{U}^{T,\psistar}_s(x)-\nabla \mathcal{U}^{T,\psistar}_s(y))\cdot\sin(\frac{\pi}{L}(x-y))}{\delta^2(x,y)}\\
\geq&\, \kappa_V(r)-\frac{4L}{\pi\,r}\|\mathcal{U}^{T,\psistar}_s\|_{\mathrm{Lip}}\geq
\bar{\kappa}_V(r) \eqsp,
\end{split}
\ees
 where 
\begin{align}
    \kappa_V(r) &\coloneqq
\inf\biggl\{-\frac{2L}{\pi}\,\frac{(\nabla V (x)-\nabla V(y))\cdot \sin(\frac{\pi}{L}(x-y))}{\delta^2(x,y)}\,: \,x\neq y\in\mathbb{T}_L^d\quad \mathrm{s.t.}\quad\delta(x,y)=r
\biggr\} \eqsp,
\end{align}
and 
\begin{align}
\label{def:bar:kappa}\bar{\kappa}_V (r)&\coloneqq\kappa_V(r)-\frac4r\,\frac{\|U_{\mu}\|_{f_V}\vee\|U_{\nu}\|_{f_V}}{1-\exp(-\lambda_V \,\pi^2\,T)}\, . 
\end{align}

By applying Proposition \ref{app:prop:contr} in Appendix \ref{app:gio} to the SOC problem \eqref{stoch:diff}, considering $\bar{\kappa}_V$ as a modulus of weak-semiconvexity for $b_s(x)\coloneqq -\nabla V(x)-\nabla \mathcal{U}^{T,\psistar}_s(x)$ (see Definition~\ref{app:new:def:kappa}), we end up with
\bes
\|\mathcal{D}_t^n\|_{\bar f_V}\leq \rme^{-\blambdaV\,\pi^2\,(T-t)}\|\psi^n-\psistar\|_{\bar f_V}\, ,
\ees
where 
\begin{equation}\label{eq:lambda_bar}
    \blambdaV  = \left(\int_0^{L \, d^{1/2}} \int_0^r  \, \exp\left(\frac{1}{2}(r^2-s^2)-\frac{1}{4}\int_s^r t\bar{\kappa}_V(t)\De t \right)\De s\eqsp\De r\right)^{-1}
\end{equation}
and $\bar{f}_V$ is a concave continuous function satisfying for any function $\phi : \torusL^d \to \rset$
\begin{equation}
\frac{1}{\pi}\,\|\phi\|_{\mathrm{Lip}}\leq \|\phi\|_{\bar{f}_V}\leq \frac{\bar{C}_V^{-1}}{\sqrt{L\eqsp \pi}}\,\|\phi\|_{\mathrm{Lip}}\eqsp, \text{ with } \bar{C}_V=\frac{e^{-L^2\,d/2}}{2}\,\exp\left(\frac14\int_0^{L\,d^{1/2}} s\,\bar{\kappa}_V(s)\,\De s\right) \eqsp .
\label{eq:equiv_lip_norm_f_bar}
\end{equation}
In order to conclude, it is enough to recall that $\varphistar,\psistar$ satisfies the Schr\"odinger system
\[\begin{cases}
\varphi=U_\mu-\mathcal{U}^{T,\psistar}_0\\
\psi=U_\nu-\mathcal{U}^{T,\varphistar}_0\,,
\end{cases}\]
which combined with Sinkhorn algorithm definition gives
\bes
\|\varphi^{n+1}-\varphistar\|_{\bar f_V}=\|\mathcal{U}^{T,\psistar}_0-\mathcal{U}^{T,\psi^n}_0\|_{\bar f_V}=\|\mathcal{D}_0^n\|_{\bar f_V}\leq \rme^{-\blambdaV\,\pi^2\,T}\|\psi^n-\psistar\|_{\bar f_V}\,.
\ees
This proves the second bound in \eqref{eq:grad_exp_2}. The first one can be proven in the same way, using the very same $\blambdaV,\bar f_V$ (thanks to the symmetrized definition for $\bar{\kappa}_V$ given at \eqref{def:bar:kappa}). The estimates in \eqref{eq:grad_exp_3} are just the two-step bounds given via the former ones.

\section{Explicit convergence rates provided in Theorem \ref{thm:sink}}\label{app:explicit:computations}
In this appendix, we carry out the computations of the convergence rates presented in Section \ref{sec:explicit_rates}. We recall from the proof of Theorem~\ref{thm:sink}, see \eqref{eq:explicit_constants}, that for any potential $V$, the convergence rate $\gamma$ and the constant $\cconst$ may be computed as $\gamma=\exp(-\blambdaV \pi^2 T)$ and  $\cconst=\bar{C}_V^{-1}/\sqrt{L\, \pi}$, where $\blambdaV$ and $\bar{C}_V$ are respectively given at \eqref{eq:lambda_bar} and \eqref{eq:equiv_lip_norm_f_bar}, and are associated to the modulus introduced at \eqref{def:bar:kappa} as
\bes
\bar{\kappa}_V (r)\coloneqq\kappa_V(r)-\frac4r\frac{\|U_{\mu}\|_{f_V}\vee\|U_{\nu}\|_{f_V}}{1-\exp(-\lambda_V \,\pi^2\,T)}\,.
\ees

Let us recall that our final goal here is to give estimates of $\gamma$ and $\cconst$ when considering a time-homogeneous drift induced by a potential $V$ (i.e, $b_s(x)=-\nabla V(x)$) satisfying a one-sided Lipschitz bound (which can be interpreted as a version of $\alpha$-semiconvexity on the torus cf. \eqref{def:sin:a:conv}) for any $ x\neq y\in\mathbb{T}_L^d$,
\bes
\sin\biggl(\frac{\pi}{L}(x-y)\biggr)^{{\sf  T}}(\nabla V(x)-\nabla V(y))\geq\frac{\pi\,\alpha}{2L}\,\delta(x,y)^2 \eqsp, 
\ees
for some $\alpha\leq 0$. The above condition trivially implies a constant (and non-positive) modulus of semi-convexity $\kappa_V= \alpha$ to which we associate via equations \eqref{def:f} and \eqref{def:lambda:C} the triplet $C_V,\lambda_V,f_V$. Particularly, we have
\bes
\begin{aligned}
\phi_V(&r)=\,\rme^{\frac{\alpha}{8}r^2},\quad C_V=\frac{\rme^{\frac{\alpha}{8}L^2\,d}}{2}\quad\text{and}\\
\lambda_V^{-1}=&\,\int_0^{Ld^{1/2}}\rme^{-\frac{\alpha}{8}r^2}\int_0^r \rme^{\frac{\alpha}{8}s^2}\De s\,\De r\leq \int_0^{Ld^{1/2}}r\,\rme^{-\frac{\alpha}{8}r^2}\,\De r\,,\text{ \ie }\,\,\lambda_V\geq \frac{|\alpha|/4}{\rme^{\frac{|\alpha|}{8}L^2d}-1}\,.
\end{aligned}
\ees
Before moving on, let us notice that in the Brownian motion case (\ie when considering $V=0$) we have $\alpha=\kappa_V=0$ and therefore
\bes
\phi_0(r)=1\,,\quad C_0=\frac12\,,\quad\lambda_0=\frac{2}{L^2\,d}\,,\text{ and }\quad f_0(r)=r-\frac{L^2\,d}{6}r^3\,,
\ees
which agrees with the above lowerbounds in the $|\alpha|$ vanishing limit.

Given the above premesis, we are now ready to compute the triplet $(\bar{C}_V, \blambdaV, \bar f_V)$ associated to $\bar{\kappa}_V$
and, for notations' sake, let us introduce 
\bes 
M\coloneqq \frac{ \|U_{\mu}\|_{f_V}\vee\|U_{\nu}\|_{f_V}}{1-\exp(-\lambda_V \,\pi^2\,T)}\geq 0\,.
\ees
Then, we immediately have $\bar{\kappa}_V(r)=\alpha-\frac4r\,M$.
By equations \eqref{def:f} and \eqref{def:lambda:C}, we deduce that
\bes
\begin{aligned}
\bar\phi(r)=&\,\rme^{\frac{\alpha}{8}r^2-Mr},\quad  \bar{C}_V=\frac{\rme^{\frac{\alpha}{8}L^2d-M Ld^{1/2}}}{2}\quad\text{and}\\
\blambdaV^{-1}=&\,\int_0^{Ld^{1/2}}\rme^{-\frac{\alpha}{8}r^2+Mr}\int_0^r \rme^{\frac{\alpha}{8}s^2-Ms}\De s\,\De r\leq e^{MLd^{1/2}}\int_0^{Ld^{1/2}}r\,\rme^{-\frac{\alpha}{8}r^2}\,\De r\,\\
\,&\text{ and hence }\,\blambdaV\geq \frac{|\alpha|/4}{\rme^{\frac{|\alpha|}{8}L^2d}-1}\,\rme^{-MLd^{1/2}}\,.
\end{aligned}
\ees
In the above estimate we have decided to bound the exponential terms depending from $M$ since in the $T$ vanishing limit (\ie\, the interesting regime that approximates the optimal transportation problem) the leading order term will be exponential in $M$, \ie equal to $\rme^{-MLd^{1/2}}$ (up to a polynomial in $M$ prefactor). In the Brownian motion case, where there is no factor $\alpha$, one may want to carry out the exact computations which lead to
\bes
\bar\lambda_0=\frac{M^2}{e^{Ld^{1/2}M}-(1+M\,L \,d^{1/2})}\eqsp.
\ees

\medskip

Then, we obtain 
\bes
\begin{split}
\log\gamma=&\,-\pi^2 T\,\blambdaV \leq  -\pi^2 T\frac{|\alpha|/4}{\rme^{\frac{|\alpha|}{8}L^2d}-1}\,\rme^{-MLd^{1/2}}\eqsp, \label{eq:log_gamma_and_c}\\
\cconst=&\,\frac{\bar{C}_V^{-1} }{\sqrt{L \,\pi}}=2\frac{\rme^{\frac{|\alpha|}{8}L^2d+M Ld^{1/2}}}{\sqrt{L\,\pi}}\eqsp,
\end{split}
\ees
whereas in the Brownian case (where we denote by $\gamma_0$ the convergence rate) we have the same constant $\cconst$ (with $|\alpha|=0$) and rate of convergence
\bes
\begin{split}
\log\gamma_0=&\,-\pi^2 T\,\bar\lambda_0 =-\pi^2\,T\,\frac{M^2}{e^{Ld^{1/2}M}-1-MLd^{1/2}}\leq -\pi^2T\,M^2\exp(-Ld^{1/2}M)\eqsp.
\end{split}
\ees

By recalling the definition of $M$, we end up with
\bes
\begin{aligned}
\log\gamma\leq&\, -\pi^2 T\frac{|\alpha|/4}{\rme^{\frac{|\alpha|}{8}L^2d}-1}\,\exp\biggl(-L\,d^{1/2}\,\frac{ \|U_{\mu}\|_{f_V}\vee\|U_{\nu}\|_{f_V}}{1-\exp(-\lambda_V \,\pi^2\,T)}\biggr)\\
\leq&\, -\pi^2 T\frac{|\alpha|/4}{\rme^{\frac{|\alpha|}{8}L^2d}-1}\,\exp\left(-L\,d^{1/2}\,\frac{ \|U_{\mu}\|_{f_V}\vee\|U_{\nu}\|_{f_V}}{1-\exp\left(-\frac{|\alpha|/{4}}{\rme^{\frac{|\alpha|}{8}L^2d}-1} \,\pi^2\,T\right)}\right)
\end{aligned}
\ees
and similarly
\begin{align}
    \cconst\leq 2\frac{\rme^{\frac{|\alpha|}{8}L^2 d}}{\sqrt{L \, \pi}}\,\exp\left(L\,d^{1/2}\frac{ \|U_{\mu}\|_{f_V}\vee\|U_{\nu}\|_{f_V}}{1-\exp\left(-\frac{|\alpha|/{4}}{\rme^{\frac{|\alpha|}{8}L^2d}-1} \,\pi^2\,T\right)}\right)\,,
\end{align}
which are the bounds given at \eqref{eq:estimate_gamma} and \eqref{eq:estimate_c}. By following the same reasoning for the Brownian case and by recalling that $\lambda_0=2/(L^2\,d)$, we end up with
\bes
\begin{aligned}
\log\gamma_0\leq&\, -\pi^2 T \,\frac{ \|U_{\mu}\|^2_{f_0}\vee\|U_{\nu}\|^2_{f_0}}{(1-\exp(-2\,\pi^2\,T/L^2d))^2}\,\exp\biggl(-L\,d^{1/2}\,\frac{ \|U_{\mu}\|_{f_0}\vee\|U_{\nu}\|_{f_0}}{1-\exp(-2 \,\pi^2\,T/L^2d)}\biggr)\\
\leq&\,- L^4\,d^2\,\frac{ \|U_{\mu}\|^2_{f_0}\vee\|U_{\nu}\|^2_{f_0}}{4\,\pi^2\,T}\,\exp\biggl(-L\,d^{1/2}\,\frac{ \|U_{\mu}\|_{f_0}\vee\|U_{\nu}\|_{f_0}}{1-\exp(-2\,\pi^2\,T/L^2d)}\biggr)
\end{aligned}
\ees
which in the small-time limit behaves as
\bes
\log\gamma_0\sim -\pi^2\,D_{\mu,\nu}^2\,{D^4}\,T^{-1}\,\exp(- D_{\mu,\nu}\,D^3\,T^{-1})\,,
\ees
where $D_{\mu,\nu}\coloneqq \frac{1}{2\pi^2}\,\|U_{\mu}\|_{f_0}\vee\|U_{\nu}\|_{f_0}$ and $D=L\,d^{1/2}$, which is exactly what claimed at \eqref{eq:asympt:gamma0}.

 \medskip

Using \eqref{eq:equiv_lip_norm}, the above bounds may also be written as
\begin{align}
    \log\gamma&\leq -\pi^2 T\frac{|\alpha|/{4}}{\rme^{\frac{|\alpha|}{8}L^2 d}-1}\,\exp\left(-2\rme^{\frac{|\alpha|}{8}L^2 d}\sqrt{\frac{L\,d}{\pi}}\frac{ \|U_{\mu}\|_{\mathrm{Lip}}\vee\|U_{\nu}\|_{\mathrm{Lip}}}{1-\exp\left(-\frac{|\alpha|/{4}}{\rme^{\frac{|\alpha|}{8}L^2d}-1} \,\pi^2\,T\right)}\right)\\
    \cconst& \leq  2\frac{\rme^{\frac{|\alpha|}{8}L^2 d}}{\sqrt{L \, \pi}}\,\exp\left(2\rme^{\frac{|\alpha|}{8}L^2 d}\sqrt{\frac{L\,d}{\pi}}\frac{ \|U_{\mu}\|_{\mathrm{Lip}}\vee\|U_{\nu}\|_{\mathrm{Lip}}}{1-\exp\left(-\frac{|\alpha|/{4}}{\rme^{\frac{|\alpha|}{8}L^2d}-1} \,\pi^2\,T\right)}\right)\,.
\end{align}

\subsection{An example with a trigonometric potential} \label{subsec:example_V}
Here we demonstrate that it is natural to assume $\alpha$-semiconvexity conditions for potentials on the torus, by showing that \eqref{def:sin:a:conv} is satisfied by a well-defined class of potentials.

Let us started by considering $(\alpha_i, \beta_i)_{i\in [d]}\in \mathbb{R}^{2d}$ and for any $i=1, ...,d$,  define $\sigma_i=~\sqrt{\alpha_i^2 +\beta_i^2}$. Without loss of generality, we reorder the dimensions such that $\sigma_1\geq ... \geq \sigma_d$.

Consider the time-homogeneous drift induced by the potential $V$ defined by 
\bes
V(x)=\frac{L}{8}\,\sum_{i=1}^d\alpha_i\,\sin\biggl(\frac{2\pi}{L} \,x^i+\omega_i\biggr)+\beta_i\,\cos\biggl(\frac{2\pi}{L} x^i+\omega_i\biggr) \quad \forall x \in \mathbb{T}_L^d ,
\ees
where $\omega_i\in \mathbb{R}$ is a phase-shifter for any $i=1, ...,d$. Then, for any $(x,y)\in\mathbb{T}_L^d \times \mathbb{T}_L^d$, we have
\begin{equation}\label{sin_eq}
\begin{aligned}
    &\sin\biggl(\frac{\pi}{L}(x-y)\biggr)^{{\sf  T}}(\nabla V(x)-\nabla V(y))= \\
    & \quad -\frac{ \pi}{2}\sum_{i=1}^d \sin^2\biggl(\frac{\pi}{L}(x^i -y^i)\biggr)\biggl\{\alpha_i \sin\biggl(\frac{\pi }{L}(x^i+y^i)+\omega_i\biggr) + \beta_i \cos\biggl(\frac{\pi}{L} (x^i+y^i)+\omega_i\biggr)\biggr\}\,.
\end{aligned}
\end{equation}
Let $r\in [0, L \, d^{1/2}]$. To be able to compute the weak-semiconvexity modulus $\kappa_V(r)$, we first expect to solve the following maximization problem
\begin{align}
 \max_{(x,y)\in\mathbb{T}_L^d \times \mathbb{T}_L^d}
    & \sum_{i=1}^d \sin^2\biggl(\frac{\pi}{L}(x^i -y^i)\biggr)\biggl\{\alpha_i \sin\biggl(\frac{\pi }{L}(x^i+y^i)+\omega_i\biggr) + \beta_i \cos\biggl(\frac{\pi}{L} (x^i+y^i)+\omega_i\biggr)\biggr\} \\
    & \text{subject to } \delta(x,y)=r\,,
\end{align}
which, by a change of variable, reads as
\begin{equation}\label{max_pb_2}
\begin{aligned}
\max_{(u,v)\in \mathbb{T}_L^d \times \mathbb{T}_L^d} &
    \sum_{i=1}^d \sin^2\biggl(\frac{\pi}{L} u^i\biggr)\biggl\{\alpha_i \sin\biggl(\frac{\pi}{L} (u^i +2v^i)+\omega_i\biggr) + \beta_i \cos\biggl(\frac{\pi}{L} (u^i +2v^i)+\omega_i\biggr)\biggr\}\\
    &  \text{subject to } L^2 \sum_{i=1}^d \sin^2\biggl(\frac{\pi}{L} u^i\biggr)=r^2\,.
\end{aligned}
\end{equation}
Since the above constraint does not depend on $v$, and for any fixed $u^i \in \mathbb{T}^1_L$  
\bes   \max_{v_i\in\mathbb{T}_L^1} \alpha_i \sin\biggl(\frac{\pi}{L} (u^i +2v^i)+\omega_i\biggr) + \beta_i \cos\biggl(\frac{\pi}{L} (u^i +2v^i)+\omega_i\biggr)=\sigma_i\,,\ees
 Problem \eqref{max_pb_2} simply reduces to
\begin{align}
    \max_{u\in \mathbb{T}_L^d} &
    \sum_{i=1}^d \sigma_i \sin^2\biggl(\frac{\pi}{L} u^i\biggr) \quad  \text{subject to } \sum_{i=1}^d \sin^2\biggl(\frac{\pi}{L} u^i\biggr)=(r/L)^2.
\end{align}
Because of our $\{\sigma_i\}_{i \in [d]}$ ordering choice (and that $\sigma_i \geq 0$ for any $i \in [d]$) we may finally conclude that the maximum value in \eqref{max_pb_2} is equal to
\begin{align}
    \sum_{i=1}^d  \sigma_i \min(1 , ((r/L)^2 -i+1)^+) \label{M(r)}.
\end{align}
By combining \eqref{M(r)} with \eqref{sin_eq} , we therefore have
\begin{align}
    \kappa_V(r)&=-\frac{L}{r^2}\sum_{i=1}^d \sigma_i \min(1, ((r/L)^2 -i+1)^+) && \text{ if } r\in (0,L \, d^{1/2}]\\
    \kappa_V(0)&=-\sigma_1/L && \text{(by continuity)}
\end{align}
Let us stress out that $s \mapsto s\kappa_V(s)$ is integrable on $(0,L \, d^{1/2}]$ and that for any $r\in [0,L \,d^{1/2}]$, we have $\kappa(r)\leq 0$, as required by Definition~\ref{app:new:def:kappa}.

\end{document}